\documentclass[a4paper,12pt]{amsart}
\usepackage[ansinew]{inputenc}
\usepackage[english]{babel} 
\usepackage[T1]{fontenc}
\usepackage{csquotes}
\usepackage{amsmath,amssymb,amsthm, mathtools, bbm, ifxetex, amsrefs, xcolor }
\usepackage{trfsigns}
\usepackage{xspace}
\usepackage{geometry}
\usepackage{comment}
\usepackage{enumerate}
\usepackage{hyperref}

\numberwithin{equation}{section}

\newtheorem{satz}{Theorem}[section]
\newtheorem{proposition}[satz]{Proposition}
\newtheorem{lemma}[satz]{Lemma}
\newtheorem{korollar}[satz]{Corollary}

\newtheorem{bemerkung}[satz]{Remark}
\newtheorem{definition}[satz]{Definition}
\newtheorem{condition}[satz]{Conditions}
\theoremstyle{definition}

\title{Higher Order Concentration of Measure}

\author{S.\,G. Bobkov}
\address{Sergey G. Bobkov,
School of Mathematics, University of Minnesota, Minneapolis, USA;
National Research University Higher School of Economics, Russian Federation
}
\email{bobkov@math.umn.edu}

\author{F. G\"otze}
\address{Friedrich G\"otze, Faculty of Mathematics, Bielefeld University, 
Bielefeld, Germany}
\email{goetze@math.uni-bielefeld.de}
\author{H. Sambale}
\address{Holger Sambale, Faculty of Mathematics, Bielefeld University, 
Bielefeld, Germany}
\email{hsambale@math.uni-bielefeld.de}

\begin{document}



\thanks{This research was supported by CRC 1283. The work of S. G. Bobkov was supported
by NSF grant DMS-1612961 and by the Russian Academic Excellence Project '5-100'. }

\begin{abstract}
We study sharpened forms of the concentration of measure phenomenon typically 
centered at stochastic expansions of order $d-1$ for any $d \in \mathbb{N}$.
The bounds are based on $d$-th order derivatives or difference operators. 
In particular, we consider deviations of functions of independent random variables
and differentiable functions over probability measures satisfying
a logarithmic Sobolev inequality, and functions on the unit sphere. 
Applications include concentration inequalities for $U$-statistics
as well as for classes of symmetric functions 
via polynomial approximations on the sphere (Edgeworth-type expansions).
\end{abstract}

\date{\today}

\maketitle

\section{Introduction}

In this article, we study higher order versions of the concentration of 
measure phenomenon. Referring to the use of derivatives or difference 
operators of higher order, say $d$, the notion of 
\textit{higher} order concentration has several aspects. In particular,
instead of the classical problem about deviations of $f$ around the mean 
$\mathbb{E}f$, one may consider potentially smaller fluctuations of 
$f - \mathbb{E}f - f_1 - \ldots - f_d$, where $f_1, \ldots, f_d$ are 
``lower order terms'' of $f$ with respect to a suitable decomposition,
such as a Taylor-type decomposition or the Hoeffding decomposition of $f$.

Starting with the works of Milman in local theory of Banach spaces, and 
of Borell, Sudakov, and Tsirelson within the framework of Gaussian processes, 
the concentration of measure phenomenon has been intensely studied during 
the past decades. This study includes important contributions due to Talagrand 
and other researchers in the 1990s, cf. Milman and Schechtman \cite{M-S},
Talagrand \cite{T}, Ledoux \cite{L1}, \cite{L2}, \cite{L3};
a more recent survey is authored by Boucheron, Lugosi and Massart \cite{B-L-M}.
 
As another previous work, let us mention Adamczak and 
Wolff \cite{A-W}, who exploited certain Sobolev-type inequalities or subgaussian
tail conditions to derive exponential tail inequalities for
functions with bounded higher-order derivatives (evaluated in terms of 
tensor-product matrix norms). While in \cite{A-W}, concentration around 
the mean is studied, the idea of sharpening concentration inequalities for 
Gaussian measures by requiring orthogonality to linear functions 
also appears in Wolff \cite{W} as well as in
Cordero-Erausquin, Fradelizi and Maurey \cite{CE-F-M}.

Our research started with second order results for functions on 
the $n$-sphere orthogonal to linear functions \cite{B-C-G},
with an approach which was continued in \cite{G-S} in presence of logarithmic 
Sobolev inequalities. This includes discrete models as well as 
differentiable functions on open subsets of $\mathbb{R}^n$. 
Here, we adapt in particular Sobolev type inequalities 
introduced by Boucheron, Bousquet, Lugosi and Massart \cite{B-B-L-M},
and thus extend some of the results from \cite{G-S} to arbitrary higher orders. 
Developing the algebra of 
higher order difference operators, we moreover came across a higher order
extension of the well-known Efron--Stein inequality.

\subsection{Functions of independent random variables}

Let $X = (X_1,\ldots,X_n)$ be a random vector in $\mathbb{R}^n$ with 
independent components, defined on some probability space 
$(\Omega, \mathcal{A}, \mathbb{P})$.
First, we state higher order exponential inequalities in terms of
the difference operator which is frequently used in the method of bounded 
differences. 

Let $(\bar{X}_1,\ldots, \bar{X}_n)$ be an independent copy of 
$X$. Given $f(X) \in L^\infty(\mathbb{P})$, define
$$
T_if(X) = T_i f = 
f(X_1, \ldots, X_{i-1}, \bar{X}_i,\linebreak[2] X_{i+1},\ldots, X_n), \qquad
i = 1, \ldots, n,
$$ 
\begin{equation}\label{h}
\mathfrak{h}_i f(X) = \frac{1}{2}\, \lVert f(X) - T_if(X) \rVert_{i, \infty},
\qquad \mathfrak{h} f = (\mathfrak{h}_1 f, \ldots, \mathfrak{h}_n f),
\end{equation}
where $\lVert \cdot \rVert_{i, \infty}$ denotes the $L^\infty$-norm with 
respect to $(X_i,\bar{X}_i)$. Depending on the random variables $X_j$, 
$j \ne i$,  $\mathfrak{h}_i f$ thus provides a uniform upper bound on 
the differences with respect to the $i$-th coordinate (up to constant).
Based on $\mathfrak{h}$, it is possible to define higher order difference operators 
$\mathfrak{h}_{i_1 \ldots i_d}$ ($d \in \mathbb{N}$) by setting
\begin{align}
\label{hd}
\begin{split}
\mathfrak{h}_{i_1 \ldots i_d}f(X) 
  = \; &
\frac{1}{2^d}\,
\Big\lVert\, \prod_{s=1}^d\, 
(Id - T_{i_s})f(X) \Big\rVert_{i_1, \ldots, i_d, \infty}\\
   = \; &
\frac{1}{2^d}\,
\Big\lVert\, f(X) + \sum_{k=1}^d\, (-1)^k \sum_{1 \leq s_1 < \ldots < s_k \leq d}
T_{i_{s_1} \ldots i_{s_k}}f(X)\, \Big\rVert_{i_1, \ldots, i_d, \infty},
\end{split}
\end{align}
where $T_{i_1 \ldots i_d} = T_{i_1} \circ \ldots \circ T_{i_d}$, and 
$\lVert \cdot \rVert_{i_1, \ldots, i_d, \infty}$ denotes the $L^\infty$-norm 
with respect to $X_{i_1}, \ldots, X_{i_d}$ and 
$\bar X_{i_1}, \ldots, \bar X_{i_d}$. For instance, 
$$
\mathfrak{h}_{ij} f = \frac{1}{4}\,
\lVert f - T_if - T_jf + T_{ij}f \rVert_{i,j,\infty} \quad 
{\rm for} \  i \ne j. 
$$
Based on \eqref{hd}, we define hypermatrices of $d$-th order differences 
as follows:
\begin{align}\label{Hessediskret}
\big(\mathfrak{h}^{(d)}f(X)\big)_{i_1 \ldots i_d} = 
\begin{cases} 
\mathfrak{h}_{i_1 \ldots i_d}f(X), & 
\text{if $i_1, \ldots, i_d$ are distinct}, \\ 0, & \text{else}. 
\end{cases}
\end{align}
For short, we freely write $\mathfrak{h}^{(d)}f$ instead of 
$\mathfrak{h}^{(d)}f(X)$. Since $T_{ii} \equiv T_i$, we necessarily have 
$\mathfrak{h}_{ii} f = \frac{1}{2}\, \mathfrak{h}_i f$. 
Therefore, removing the $d$-th order differences in which some indexes appear 
more than once can be interpreted as removing lower order differences. 

Moreover, define 
$|\mathfrak{h}^{(d)}f|_\mathrm{HS}$ to be the Euclidean norm of 
$\mathfrak{h}^{(d)}f$ regarded as an element of $\mathbb{R}^{n^d}$. 
For instance, $|\mathfrak{h}^{(1)}f|_\mathrm{HS}$ is
the Euclidean norm of $\mathfrak{h} f$, and 
$|\mathfrak{h}^{(2)}f|_\mathrm{HS}$ is the Hilbert--Schmidt norm of the 
``Hessian'' $\mathfrak{h}^{(2)}f$. Also, put
\begin{equation}\label{LpNorm1}
\lVert \mathfrak{h}^{(d)}f \rVert_{\mathrm{HS},p} = 
\left( \mathbb{E}\, |\mathfrak{h}^{(d)}f|_\mathrm{HS}^p \right)^{1/p}, \qquad
p \in (0, \infty].
\end{equation}
Using these notations, the following result holds for any fixed integer $d = 1,\dots,n$.

\vskip5mm
\begin{satz}\label{allgem}
Let $f = f(X)$ be in $L^\infty(\mathbb{P})$ with $\mathbb{E} f = 0$. 
If the conditions
\begin{equation}\label{Cond1}
\lVert \mathfrak{h}^{(k)} f \rVert_{\mathrm{HS}, 2} \le 1\qquad 
(k = 1, \ldots, d-1)
\end{equation}
and
\begin{equation}\label{Cond2}
\lVert\mathfrak{h}^{(d)} f\rVert_{\mathrm{HS},\infty} \le 1
\end{equation}
are satisfied, then, for some universal constant 
$c > 0$,
$$
\mathbb{E} \exp\left(c\, |f|^{2/d}\right) \le 2.
$$
\end{satz}

\vskip2mm
Here, a possible choice is $c = 1/(208\, e)$.

In the case $d=1$, \eqref{Cond1} does not contain any constraint, while
\eqref{Cond2} means the boundedness of fluctuations of $f$ along every 
coordinate. Here we arrive at the well-known assertion on  Gaussian 
deviations of $f(X)$. For growing $d$, the conclusion is getting somewhat 
weaker, however it holds as well under the potentially much weaker 
assumptions \eqref{Cond1}--\eqref{Cond2}. To interpret them in case 
$d \geq 2$, let us recall the notion of a Hoeffding decomposition, 
introduced by Hoeffding in 1948 \cite{Hoe}. Given a function 
$f(X) \in L^1(\mathbb{P})$, it is the unique decomposition
\begin{align}\label{Hoeffding}
f(X_1, \ldots, X_n) 
 & = 
\mathbb{E} f(X) + \sum_{1 \leq i \leq n} h_i(X_i) + 
\sum_{1 \leq i < j \leq n} h_{ij}(X_i, X_j) + \ldots\\
 & = 
f_0 + f_1 + f_2 + \ldots + f_n\notag
\end{align}
such that $\mathbb{E}_{i_s} h_{i_1 \ldots i_k}(X_{i_1}, \ldots, X_{i_k}) = 0$ 
whenever $1 \le i_1 < \ldots < i_k \le n$, $s = 1, \ldots, k$, 
where $\mathbb{E}_i$ denotes the expectation with respect to $X_i$.
The sum $f_d$ is called Hoeffding term of degree $d$ or simply $d$-th 
Hoeffding term of $f$. Provided that $f(X) \in L^2(\mathbb{P})$, the system
$\{f_i(X)\}_{i = 0}^n$ forms an orthogonal decomposition of $f(X)$ 
in $L^2(\mathbb{P})$.

It is not hard to see that
$
\mathfrak{h}_{i_1 \ldots i_k} f = \mathfrak{h}_{i_1 \ldots i_k} \big(\sum_{i=k}^{n} f_i\big)
$
whenever $i_1 < \ldots < i_k$, $k \leq d$.
In this sense, \eqref{Cond1} controls the lower order Hoeffding terms 
$f_1, \ldots, f_{d-1}$, while the behaviour of $\sum_{i=d}^{n} f_i$ is mainly
controlled by \eqref{Cond2}.
The relationship between these two conditions may be 
illustrated by considering a special class  of functions
$f$ like \emph{multilinear polynomials}, that is
\begin{align}
\label{multilinear}
f(X_1, \ldots, X_n) 
 & = 
\alpha_0 + \sum_{i=1}^n \alpha_{i} X_i + \sum_{i<j} \alpha_{ij} X_i X_j + \ldots\\
 & = 
f_0 + f_1 + f_2 + \ldots + f_n \qquad \quad (\alpha_I \in \mathbb{R}).
\notag
\end{align}

\vskip5mm
\begin{proposition}
\label{multPol}
Let $X_1, \ldots, X_n$ be bounded and such that $\mathbb{E} X_i = 0$, 
$\mathbb{E} X_i^2 = 1$ for $i = 1, \ldots, d$.
If $f(X)$ is a multilinear polynomial \eqref{multilinear} of the form
$
f = \sum_{k=d}^{n} f_k,
$
then \eqref{Cond2} implies \eqref{Cond1}.
\end{proposition}

\vskip2mm
Note that under the conditions of Proposition \ref{multPol}, the Hoeffding 
decomposition of $f$ can be read off the polynomial structure:
$h_{i_1 \ldots i_k} (X_{i_1}, \ldots, X_{i_k}) = 
\alpha_{i_1 \ldots i_k} X_{i_1} \cdots X_{i_k}$.

In particular, let $X_1, \ldots, X_n$ be independent Rademacher variables, i.\,e. 
$X_i$'s have distribution $\frac{1}{2} \delta_{+1} + \frac{1}{2} \delta_{-1}$,
where $\delta_x$ denotes the Dirac measure in $x$. In this case, 
any function $f(X)$ can be written as a multilinear polynomial with coefficients 
$$
\alpha_{i_1 \ldots i_k} = \mathbb{E} f(X_1, \ldots, X_n)\,X_{i_1} \cdots X_{i_k},
\qquad 1 \leq i_1 < \ldots < i_k \leq n.
$$
This representation is known as \emph{Fourier--Walsh expansion} of $f$.
Consequently, Theorem \ref{allgem} yields a $d$-th order concentration result 
on the discrete hypercube, and if $f$ has Fourier--Walsh expansion of type
$f = \sum_{k=d}^{n} f_k$ satisfying \eqref{Cond2}, the concentration bound 
given in Theorem \ref{allgem} holds.

Using Rademacher variables in Theorem \ref{allgem} gives rise to a concentration 
inequality for $U$-statistics with completely degenerate kernel functions. 
There are many results on the distributional properties of $U$-statistics 
(cf. de la Pe\~{n}a and Gin\'e \cite{D-G} for an overview). Starting 
with Hoeffding's inequalities (e.\,g. \cite{D-G}, Theorem 4.1.8), we especially 
refer to the results by Arcones and Gin\'e \cite{A-G} and Major \cite{M}.
By combining elements of the proof of Theorem \ref{allgem} and a classical result 
on randomized $U$-statistics by de la Pe\~{n}a and Gin\'e, we
arrive at the following:

\vskip5mm
\begin{korollar}
\label{U-statistics}
Let $X_1, \ldots, X_n$ be i.i.d. random elements in a measurable space 
$(S, \mathcal{S})$, and $h$ be a measurable function on $S^d$ 
$(1 \leq d \leq n)$ such that $|h| \leq M$ for some constant $M$. 
If $h$ is completely degenerate, i.\,e. 
$\mathbb{E}_i\, h(X_1, \ldots, X_d) = 0$ for all $i = 1, \ldots, d$, 
then, for some positive constant $c = c(d,M)$, the $U$-statistic
$$
f(X_1, \ldots, X_n) \, = \, \frac{(n-d)!}{n!} 
\sum_{i_1 \ne \ldots \ne i_d} h(X_{i_1}, \ldots, X_{i_d})
$$
satisfies
$$
\mathbb{E} \exp\left(c n\, |f|^{2/d}\right) \, \le \, 2.
$$
\end{korollar}

\vskip2mm
By Chebychev's inequality, Theorem \ref{allgem} immediately
yields a deviation bound
$$
\mathbb{P}\{|f(X)| \ge t\} \, \le \, 2 e^{-ct^{2/d}}, \quad t \ge 0.
$$
More precisely, we get refined tail estimates similar to 
Adamczak \cite{A}, Theorem 7, or Adamczak and Wolff \cite{A-W}, 
Theorem 3.3.

\vskip5mm
\begin{korollar}
\label{KorrTails}
Let $f = f(X)$ be in $L^\infty(\mathbb{P})$ with $\mathbb{E} f = 0$. For all
$t \geq 0$, putting
$$
\eta_f(t) \, = \, 
\min \Big(\frac{t^{2/d}}{\lVert \mathfrak{h}^{(d)} 
f \rVert_{\mathrm{HS}, \infty}^{2/d}}, \min_{k=1, \ldots, d-1}
\frac{t^{2/k}}{\lVert\mathfrak{h}^{(k)}f \rVert_{\mathrm{HS},2}^{2/k}} \Big),
$$
we have
$$
\mathbb{P}\{|f| \ge t\} \, \le \, e^2\, \exp\{-\eta_f(t)/ 41 (de)^2\}.
$$
\end{korollar}

\vskip2mm
Moreover, it is possible to give a version of Theorem \ref{allgem} for suprema of
suitable classes of functions. To this end, we need some more notation. Let
$\mathfrak{F}$ be a class of functions $f = f(X)$ in
$L^\infty(\mathbb{P})$, where as before $X = (X_1, \ldots, X_n)$ is a vector of independent
random variables. Then, for $i_1 \ne \ldots \ne i_d$, $d = 1, \ldots, n$, we define
$$
\mathfrak{h}^*_{i_1 \ldots i_d}(\mathfrak{F})  \, =  \, \sup_{f \in \mathfrak{F}}\,
\mathfrak{h}_{i_1 \ldots i_d} f(X)
$$
as a structural supremum (eventually taken over a countable subset of $\mathfrak{F}$), 
and put
$$
\big(\mathfrak{h}^{*(d)}(\mathfrak{F})\big)_{i_1 \ldots i_d} = 
\begin{cases} 
\mathfrak{h}^*_{i_1 \ldots i_d}(\mathfrak{F}), & 
\text{if $i_1, \ldots, i_d$ are distinct}, \\ 0, & \text{else}. 
\end{cases}
$$
The notations used in \eqref{LpNorm1} are similarly adapted. This leads to the following
result.

\vskip5mm
\begin{satz}\label{allgemsup}
If $\lVert \mathfrak{h}^{*(k)}(\mathfrak{F}) \rVert_{\mathrm{HS}, 2} \le 1$ for all 
$k = 1, \ldots, d-1$ and
$\lVert\mathfrak{h}^{*(d)}(\mathfrak{F})\rVert_{\mathrm{HS},\infty} \le 1$,
then
$$
\mathbb{E} \exp\Big\{c\, \big|\sup_{f \in \mathfrak{F}}|f| - \mathbb{E}
\sup_{f \in \mathfrak{F}}|f|\big|^{2/d}\Big\} \le 2
$$
with some universal constant $c > 0$.
\end{satz}

\vskip2mm
Finally, to provide another application, recall the example of additive functionals
of partial sums (e.\,g. random walks)
\begin{equation}
\label{partsum}
S_f = S_f(X) = \sum_{i=1}^{n} f\Big(\sum_{j=1}^{i} X_j\Big).
\end{equation}
In \cite{G-S} we proved a second order concentration result for functionals of this
type, which may be reproved and somewhat sharpened by applying Corollary
\ref{KorrTails}:

\vskip5mm
\begin{proposition}
\label{proppartsum}
Given a bounded, Borel measurable function $f \colon \mathbb{R} \to \mathbb{R}$, 
for any $t \ge 0$,
$$
\mathbb{P}(|S_f - \mathbb{E}S_f| \ge t) \, \le \, 
e^2\, \exp \Big\{-c \min\Big(\frac{t^2}{n^3\lVert f \rVert_\infty^2},
\frac{t}{n^2 \lVert f \rVert_\infty}\Big)\Big\},
$$
where $c > 0$ is some numerical constant.
\end{proposition}

\vskip5mm
\subsection{Higher order Efron--Stein inequality}

Given independent random variables $X_1, \ldots, X_n$, 
we denote by $\mathbb{E}_i\, f(X) = \mathbb{E}_i f$ and
$\text{Var}_i f(X) = \mathbb{E}_i\, (f(X) - \mathbb{E}_i f(X))^2$
the expected value and variance with respect to $X_i$. By a well-known 
result of Efron and Stein \cite{E-S}, the variance functional is subadditive
in the sense that
\begin{equation}\label{Efron-Stein}
\mathrm{Var} f(X) \le \mathbb{E} \sum_{i=1}^{n} \text{Var}_i\,f(X).
\end{equation}
It is possible to restate \eqref{Efron-Stein} in terms of difference operators 
which we introduce below. As before, let $\bar{X}_1,\ldots, \bar{X}_n$ 
be a set of independent copies of $X_1, \ldots, X_n$ and 
$T_i f = f(X_1, \ldots, X_{i-1}, \bar{X}_i,\linebreak[2] X_{i+1},\ldots, X_n)$.
We use $\bar{\mathbb{E}}_i$ to denote the expectation with respect to $\bar{X}_i$,
and put $x_+ = \max(x,0)$ and $x_- = \max(-x,0)$ for a number~$x$.

\vskip5mm
\begin{definition}
\label{DiffOp}
Let $f(X) = f(X_1, \ldots, X_n)$ be a measurable function. For
$i = 1, \ldots, n$, under proper integrability assumptions, put:

\begin{enumerate}[(i)]
\item
$$
\mathfrak{v}_if(X) = 
\big(\mathrm{Var}_i\, f(X)\big)^{1/2}, \qquad \qquad
\mathfrak{v} f = (\mathfrak{v}_1f, \ldots, \mathfrak{v}_nf);
$$
\item
$$
\mathfrak{D}_i f(X) = f(X) - \mathbb{E}_i f(X), \qquad \qquad
\mathfrak{D}f = (\mathfrak{D}_1f, \ldots, \mathfrak{D}_nf);
$$
\item
$$
\mathfrak{d}_i f(X) = 
\Big(\frac{1}{2}\,\bar{\mathbb{E}}_i(f(X) - T_if(X))^2\Big)^{1/2}, \qquad
\mathfrak{d} f = (\mathfrak{d}_1f, \ldots, \mathfrak{d}_n f);
$$
\item
$$
\mathfrak{d}^+_i f(X) = 
\Big(\frac{1}{2}\,\bar{\mathbb{E}}_i(f(X) - T_if(X))_+^2\Big)^{1/2}, \qquad
\mathfrak{d}^+f = (\mathfrak{d}^+_1f, \ldots, \mathfrak{d}^+_nf);
$$
\item
$$
\mathfrak{d}^-_if(X) = 
\Big(\frac{1}{2}\,\bar{\mathbb{E}}_i(f(X) - T_if(X))_-^2\Big)^{1/2}, \qquad
\mathfrak{d}^-f = (\mathfrak{d}^-_1f, \ldots, \mathfrak{d}^-_nf).
$$
\end{enumerate}
\end{definition}

\vskip2mm
Various relations between these difference operators are
discussed in Section 3. In particular, it is easy to see that,
for $f(X) \in L^2(\mathbb{P})$,
\begin{equation}
\label{L2normengleich}
\mathbb{E}\,|\mathfrak{v}f|^2 = \mathbb{E}\,|\mathfrak{D}f|^2 = 
\mathbb{E}\,|\mathfrak{d}f|^2 = 2\, \mathbb{E}\,|\mathfrak{d}^+f|^2 = 
2\, \mathbb{E}\,|\mathfrak{d}^-f|^2,
\end{equation}
where $| \cdot |$ denotes the Euclidean norm in $\mathbb{R}^n$.
Therefore, we may equivalently state the Efron--Stein inequality as
\begin{gather}
\label{Efron-SteinwDiffOp}
\begin{split}
\text{Var}f \le \mathbb{E}\,|\mathfrak{v}f|^2,\qquad 
\text{Var}f \le \mathbb{E}\,|\mathfrak{D}f|^2,\qquad 
\text{Var}f \le \mathbb{E}\,|\mathfrak{d}f|^2,\\
\text{Var}f \le 2\, \mathbb{E}\,|\mathfrak{d}^+f|^2\qquad \text{or}\qquad 
\text{Var}f \le 2\, \mathbb{E}\,|\mathfrak{d}^-f|^2.
\end{split}
\end{gather}
Equality in \eqref{Efron-SteinwDiffOp} holds 
iff the Hoeffding decomposition of $f$ consists of the expected value and the first
order term only, namely for $f(X) = \mathbb{E} f(X) + \sum_{i=1}^{n} h_i(X_i)$. 
Thus, the Efron--Stein inequality
may be restated as the fact that any product probability measure 
satisfies a Poincar\'e-type inequality with respect to any of the difference
operators $\mathfrak{v}$, $\mathfrak{D}$ and $\mathfrak{d}$ with constant 
$\sigma^2 = 1$ (like \eqref{PI1} below). The same statement applies 
as well to the difference operators $\mathfrak{d}^+$ and $\mathfrak{d}^-$ 
but with constant $\sigma^2 = 2$.

To introduce higher order versions of the Efron--Stein inequality, 
we need to define higher order differences based on the difference operators 
from Definition \ref{DiffOp}. For $\mathfrak{D}$, this is achieved 
by iteration, i.\,e. $\mathfrak{D}_{ij} f = \mathfrak{D}_i(\mathfrak{D}_j f)$ 
or, in general, 
$\mathfrak{D}_{i_1 \ldots i_d} f = \mathfrak{D}_{i_1}(\ldots (\mathfrak{D}_{i_d}f))$ 
for $1 \le i_1, \ldots, i_d \le n$.
To generalize $\mathfrak{v}$, we set similarly to \eqref{hd}
\begin{align}
\label{vd}
\begin{split}
\mathfrak{v}_{i_1 \ldots i_d}f(X) 
 = \; &
\Big(\mathbb{E}_{i_1 \ldots i_d}\,
\Big(\prod_{s=1}^d\, (Id - \mathbb{E}_{i_s})\, f(X)\Big)^2\,\Big)^{1/2}\\
  = \; &
\Big(\mathbb{E}_{i_1 \ldots i_d}\,
\Big(f(X) + \sum_{k=1}^d\, (-1)^k 
\sum_{1 \leq s_1 < \ldots < s_k \leq d} 
\mathbb{E}_{i_{s_1} \ldots i_{s_k}}f(X)\Big)^2\,\Big)^{1/2},
\end{split}
\end{align}
Here, $\mathbb{E}_{i_1 \ldots i_d}$ means taking the expectation with respect 
to $X_{i_1}, \ldots, X_{i_d}$. For instance,
$$
\mathfrak{v}_{ij} f = \Big(\mathbb{E}_{ij}\, 
(f - \mathbb{E}_i f - \mathbb{E}_j f + \mathbb{E}_{ij}f)^2\Big)^{1/2}, \qquad
1 \leq i < j \leq n.
$$ 
In particular, $\mathfrak{v}_{ij} f \ne (\text{Var}_{ij}f)^{1/2}$.
One major difference is that $\mathfrak{v}_{ij} f$ annihilates first order 
Hoeffding terms, but $\text{Var}_{ij}f$ does not. Similar remarks hold for any 
$d \ge 2$.

Finally, in case of $\mathfrak{d}$, we define
\begin{align}
\label{dd}
\begin{split}
\mathfrak{d}_{i_1 \ldots i_d} f(X) 
 = \; &
\Big(\frac{1}{2^d}\,\bar{\mathbb{E}}_{i_1 \ldots i_d}\,
\Big(\prod_{s=1}^d\, (Id - T_{i_s})\,f(X)\Big)^2\,\Big)^{1/2}\\
  = \; &
\Big(\frac{1}{2^d}\,\bar{\mathbb{E}}_{i_1 \ldots i_d}\,
\Big(f(X) + \sum_{k=1}^d (-1)^k \sum_{1 \leq s_1 < \ldots < s_k \leq d} 
T_{i_{s_1} \ldots i_{s_k}} f(X)\Big)^2\,\Big)^{1/2}.
\end{split}
\end{align}
Here, $\bar{\mathbb{E}}_{i_1 \ldots i_d}$ means taking the expectation with 
respect to $\bar{X}_{i_1}, \ldots, \bar{X}_{i_d}$, recalling that 
$T_{i_1 \ldots i_d} = T_{i_1} \circ \ldots \circ T_{i_d}$. For 
$\mathfrak{d}^{\pm}$, a variant of \eqref{dd} 
holds by setting
\begin{equation*}
\mathfrak{d}_{i_1 \ldots i_d}^\pm f(X) = 
\Big(\frac{1}{2^d}\,\bar{\mathbb{E}}_{i_1 \ldots i_d}\,
\Big(\prod_{s=1}^d\, (Id - T_{i_s})f(X)\Big)_\pm^2\Big)^{1/2}.
\end{equation*}

In the same way as in \eqref{Hessediskret}, we may define $d$-th order 
hyper-matrices with respect to any of the difference operators introduced 
above, e.\,g.
\begin{align*}
(\mathfrak{v}^{(d)}f)_{i_1 \ldots i_d} 
 & = 
\begin{cases} 
\mathfrak{v}_{i_1 \ldots i_d}f, 
 & \text{if $i_1, \ldots, i_d$ are distinct}, \\ 0, 
 & \text{else}. 
\end{cases}
\end{align*}
The hyper-matrices $\mathfrak{D}^{(d)} f$, $\mathfrak{d}^{(d)} f$ and 
$\mathfrak{d}^{\pm (d)} f$ are defined analogously. As in case of 
$\mathfrak{h}^{(d)} f$, we equip these hyper-matrices with the respective 
Hilbert--Schmidt type norms.
We are now ready to formulate the following generalization of 
\eqref{Efron-SteinwDiffOp}.

\vskip5mm
\begin{satz}[Higher Order Efron--Stein Inequality]
\label{HigherOrderEfronStein}
Let $X_1, \ldots, X_n$ be independent random variables, 
and assume that $f(X) \in L^2(\mathbb{P})$ admits a Hoeffding decomposition 
of type
$
f = \mathbb{E}f + \sum_{k=d}^{n} f_k
$
for some $1 \leq d \leq n$. Then
$$
\mathrm{Var} f \le \frac{1}{d!}\, \mathbb{E}\, |\mathfrak{v}^{(d)} f|^2,\qquad 
\mathrm{Var} f \le \frac{1}{d!}\, \mathbb{E}\, |\mathfrak{D}^{(d)} f|^2,\qquad
\mathrm{Var} f \le \frac{1}{d!}\, \mathbb{E}\, |\mathfrak{d}^{(d)} f|^2.
$$
Moreover,
$$
\mathrm{Var} f \le \frac{2}{d!}\, \mathbb{E}\, |\mathfrak{d}^{+(d)} f|^2 \qquad
\text{and}\qquad 
\mathrm{Var} f \le \frac{2}{d!}\, \mathbb{E}\, |\mathfrak{d}^{-(d)} f|^2.
$$
Equality holds if and only if the Hoeffding decomposition of $f$ consists of 
the expected value and the $d$-th order term only, i.e. $f = \mathbb{E}f + f_d$.
\end{satz}

\vskip2mm
In particular, Theorem \ref{HigherOrderEfronStein} yields the
following formula for the variance of an arbitrary function $f = f(X) \in
L^2(\mathbb{P})$ with Hoeffding decomposition $f = \sum_{k=0}^{n} f_k$:
\begin{equation}
\label{Varianzdarst}
\mathrm{Var} f = \sum_{k=1}^{n} \frac{1}{k!} \mathbb{E}\, \lvert
\mathfrak{v}^{(k)} f_k\rvert^2 = \sum_{k=1}^{n} \frac{1}{k!} \mathbb{E}\,
\lvert \mathfrak{D}^{(k)} f_k\rvert^2 = \sum_{k=1}^{n} \frac{1}{k!}
\mathbb{E}\, \lvert \mathfrak{d}^{(k)} f_k\rvert^2.
\end{equation}
This result is related to the work of Houdr\'e \cite{Hou}, who studied
iterations of the Efron--Stein inequality for symmetric functions in
the context of the jackknife estimate of the variance. In particular,
he obtained formulas for the variance in terms of certain higher order
difference operators adapted to this situation.
Following the lines of our proofs, it is possible to extend his results
to arbitrary functions of independent random variables. To provide an
example, we may prove that
\begin{equation}
\label{Houdre-Formel}
\mathrm{Var} f = \sum_{k=1}^{n} \frac{(-1)^{k+1}}{k!} \mathbb{E}
\lvert \mathfrak{v}^{(k)} f \rvert^2,
\end{equation}
which is an extension of (1.3) from \cite{Hou}. As always, here the
difference operator $\mathfrak{v}$ can be replaced by $\mathfrak{D}$,
$\mathfrak{d}$ and (up to a factor $2$) $\mathfrak{d}^\pm$.

\vskip5mm
\subsection{Differentiable Functions}

In the following we shall develop higher order concentration in the setting 
of smooth functions on $\mathbb{R}^n$. Here we 
may derive similar results in the spirit of Adamczak and Wolff \cite{A-W}, 
when the underlying probability measure satisfies a logarithmic Sobolev 
inequality. Let us recall that a Borel probability measure $\mu$
on an open set $G \subset \mathbb{R}^n$ is said to satisfy 
a \emph{Poincar\'e-type} and respectively 
a \emph{logarithmic Sobolev inequality} with constant $\sigma^2 > 0$, 
if for any bounded smooth function $f$ on $G$ with gradient $\nabla f$,
respectively
\begin{equation}
\label{PI1}
\text{Var}_\mu (f) \le \sigma^2 \int |\nabla f|^2\, d\mu,
\end{equation}
\begin{equation}
\label{LSI1}
\text{Ent}_\mu (f^2) \le 2 \sigma^2 \int |\nabla f|^2\, d\mu.
\end{equation}
Here, $\text{Var}_\mu (f) = \int f^2\, d\mu - (\int f\, d\mu)^2$ is the
variance, and 
$\text{Ent}_\mu (f^2) = 
\int f^2 \log f^2\, d\mu - \int f^2\, d\mu\ \log\int f^2\, d\mu$
is the entropy functional.
Logarithmic Sobolev inequalities are stronger than Poincar\'e inequalities,
in the sense that \eqref{LSI1} implies \eqref{PI1}.

Given a function $f \in {C}^d(G)$,
we define $f^{(d)}$ to be the (hyper-) matrix whose entries
\begin{equation}
\label{Hesseallgem}
f^{(d)}_{i_1 \ldots i_d}(x) = \partial_{i_1 \ldots i_d} f(x), \qquad
d = 1,2,\dots
\end{equation}
represent the $d$-fold (continuous) partial derivatives of $f$ at $x \in G$. 
By considering $f^{(d)}(x)$ as a symmetric multilinear $d$-form, we define 
operator-type norms by
\begin{equation}
\label{Operatornorm}
|f^{(d)}(x)|_\mathrm{Op} = 
\sup \left\{ f^{(d)}(x)[v_1, \ldots, v_d] \colon |v_1| = \cdots = |v_d| = 1\right\}.
\end{equation}
For instance, $|f^{(1)}(x)|_\mathrm{Op}$ is the Euclidean norm of 
the gradient $\nabla f(x)$, and $|f^{(2)}(x)|_\mathrm{Op}$ is the operator norm 
of the Hessian $f''(x)$. Furthermore, similarly to \eqref{LpNorm1}, we will use 
the short-hand notation
\begin{equation}
\label{Lpnorm2}
\lVert f^{(d)} \rVert_{\mathrm{Op}, p} =
\left(\int_G |f^{(d)}|_\mathrm{Op}^p\, d\mu\right)^{1/p}, \qquad p \in (0, \infty].
\end{equation}

We now have the following results, assuming that $\mu$ is a probability measure 
on $G$ satisfying a logarithmic Sobolev inequality with constant $\sigma^2 > 0$.

\vskip5mm
\begin{satz}
\label{kontinuierlich}
Let $f \colon G \to \mathbb{R}$ be a $\mathcal{C}^d$-smooth function with 
$\int_G f\, d\mu = 0$. If
\begin{equation}
\label{Bed1}
\lVert f^{(k)} \rVert_{\mathrm{Op},2} \le 
\min(1, \sigma^{d-k})\qquad \forall k = 1, \ldots, d-1
\end{equation}
and
\begin{equation}
\label{Bed2}
\lVert f^{(d)} \rVert_{\mathrm{Op}, \infty} \le 1,
\end{equation}
then with some universal constant $c > 0$ we have
$$
\int_G \exp\Big\{\frac{c}{\sigma^2}\, |f|^{2/d}\Big\} d\mu \, \le \, 2.
$$
\end{satz}

\vskip2mm
Here, a possible choice is $c = 1/(8e)$.
If $f$ has centered partial derivatives of order up to $d-1$, 
it is possible to replace \eqref{Bed1} by a possibly simpler condition. 
To this end, as in the previous section, we need to involve 
Hilbert--Schmidt-type norms $|f^{(d)}(x)|_\mathrm{HS}$ 
which are defined by taking the Euclidean 
norm of $f^{(d)}(x) \in \mathbb{R}^{n^d}$. As in \eqref{LpNorm1}, 
$\lVert f^{(d)}\rVert_{\mathrm{HS}, 2}$ then denotes the $L^2$-norm
of $|f^{(d)}|_\mathrm{HS}$. In detail:

\vskip5mm
\begin{satz}
\label{kontinuierlichmAbl}
Let $f \colon G \to \mathbb{R}$ be a $\mathcal{C}^d$-smooth function 
such that $\int_G f\, d\mu = 0$ and $\int_G \partial_{i_1 \ldots i_k} f\, d\mu = 0$
for all $k = 1, \ldots, d-1$ and $1 \le i_1, \ldots, i_k \le n$.
Assume that
$$
\lVert f^{(d)} \rVert_{\mathrm{HS}, 2} \le 1\qquad \text{and}\qquad 
\lVert f^{(d)} \rVert_{\mathrm{Op}, \infty} \le 1.
$$
Then, there exists some universal constant $c > 0$ such that
$$
\int_G \exp \Big\{\frac{c}{\sigma^2}\, |f|^{2/d}\Big\} d\mu \le 2.
$$
\end{satz}

\vskip2mm
Here again, a possible choice is $c = 1/(8e)$.
Note that, by partial integration, if $\mu$ is the standard Gaussian 
measure, the conditions $\int_G f d\mu = 0$ and 
$\int_G \partial_{i_1 \ldots i_k} f d\mu = 0$ are satisfied if $f$ is 
orthogonal to all polynomials of (total) degree at most $d-1$.

As in case of functions of independent random variables, it is possible to refine
the tail estimates implied by Theorem \ref{kontinuierlich}:

\vskip5mm
\begin{korollar}
\label{KorrTailskont}
Let $f \colon G \to \mathbb{R}$ be a $\mathcal{C}^d$-smooth function such that
$\int_G f\, d\mu = 0$. For any $t \ge 0$, putting
$$
\eta_f(t) \, = \, \min \Big(\frac{t^{2/d}}{\sigma^2 \lVert f^{(d)} \rVert_{\mathrm{Op},
\infty}^{2/d}}, \min_{k=1, \ldots, d-1}
\frac{t^{2/k}}{\sigma^2 \lVert f^{(k)} \rVert_{\mathrm{Op},2}^{2/k}} \Big),
$$
we have
$$
\mu(|f| \ge t) \, \le \, e^2 \exp\{-\eta_f(t)/(de)^2\}.
$$
\end{korollar}

\vskip2mm
Note that for $d = 2$ and functions $f(X) = \sum_{i,j} a_{ij} X_iX_j$, where
$X_1, \ldots, X_n$ are independent with mean zero, this yields 
Hanson--Wright type inequalities.

Possible applications of Theorem \ref{kontinuierlich} include functionals of the
eigenvalues of random matrices. As in \cite{G-S}, we consider two situations. 
First, let $\{\xi_{jk}\}_{1 \le j \le k \le N}$ be a family of independent random 
variables, and assume that the distributions of the $\xi_{jk}$'s all satisfy 
a (one-dimensional) logarithmic Sobolev inequality \eqref{LSI1} with common 
constant $\sigma^2$. Putting $\xi_{jk} = \xi_{kj}$ for $k < j$, consider 
a symmetric $N \times N$ random matrix $\Xi = (\xi_{jk}/\sqrt{N})_{1 \le j, k \le N}$ 
and denote by $\mu^{(N)} = \mu$ the joint distribution of its ordered eigenvalues 
$\lambda_1 \le \ldots \le \lambda_N$ on $\mathbb{R}^N$ (note that 
$\lambda_1 < \ldots < \lambda_N$ a.s.).

Secondly, we consider $\beta$-ensembles: for $\beta > 0$ fixed, let
$\mu_{\beta, V}^{(N)} = \mu^{(N)} = \mu$ be the probability distribution on
$\mathbb{R}^N$ with density given by
\begin{equation}
\label{beta-ensemble}
\mu (d \lambda) = \frac{1}{Z_N} e^{-\beta N \mathcal{H}(\lambda)} d\lambda,\quad
\mathcal{H}(\lambda) = \frac{1}{2} \sum_{k=1}^{N} V(\lambda_k) - \frac{1}{N}
\sum_{1 \le k < l \le N} \log(\lambda_l - \lambda_k)
\end{equation}
for $\lambda = (\lambda_1, \ldots, \lambda_N)$, 
$\lambda_1 < \ldots < \lambda_N$. Here, $V \colon \mathbb{R} \to \mathbb{R}$ 
is a strictly convex $\mathcal{C}^2$-smooth function, and $Z_N$ is a normalization 
constant. For $\beta = 1, 2, 4$, these probability measures correspond to 
the distributions of the classical invariant random matrix ensembles (orthogonal, 
unitary and symplectic, respectively). For other $\beta$, one can interpret 
\eqref{beta-ensemble} as particle systems on the real line with Coulomb interactions.

In both cases, the probability measure $\mu$ satisfies a logarithmic Sobolev
inequality with constant of order $1/N$ (see \cite{G-S} for details). Throughout the
rest of this section, we consider the probability space $(\mathbb{R}^N,
\mathbb{B}^N, \mu)$, where $\mu$ is one of the two probability measures introduced
above, supported on the set $\lambda_1 < \ldots < \lambda_N$.

In \cite{G-S}, we studied concentration bounds for linear and quadratic eigenvalue
statistics. Those results may be reproved (up to constants) using Theorem
\ref{kontinuierlichmAbl}, and by Corollary \ref{KorrTailskont} it is moreover
possible to give slightly more accurate estimates for the tails. In the sequel, we
will rather study a related problem, namely multilinear polynomials in the
eigenvalues $\lambda_1, \ldots, \lambda_N$. That is, we consider functionals of type
\begin{equation}
\label{multilinPol}
\sum_{i_1 \ne \ldots \ne i_d} a_{i_1 \ldots i_d} \lambda_{i_1} \cdots \lambda_{i_d}.
\end{equation}
Here, $a_{i_1 \ldots i_d}$ are real numbers such that for any permutation $\sigma
\in S^d$, $a_{\sigma(i_1) \ldots \sigma(i_d)} \equiv a_{i_1 \ldots i_d}$, and
$a_{i_1 \ldots i_d} = 0$ whenever the indexes $i_1, \ldots, i_d$ are not pairwise
different. This gives rise to a hypermatrix $A = (a_{i_1 \ldots i_d}) \in
\mathbb{R}^{n^d}$, whose Euclidean norm we denote by $\lVert A
\rVert_{\mathrm{HS}}$. Moreover, set 
$\lVert A \rVert_\infty = \max_{i_1 < \ldots < i_d} \, \lvert a_{i_1 \ldots i_d}\rvert$.

According to the framework sketched in Theorem \ref{kontinuierlichmAbl}, we shall not
only center around the expected value of \eqref{multilinPol} but also around some
``lower order'' terms in order to arrive at centered derivatives of order up to
$d-1$. We work out details for $d = 1, \ldots, 4$. To facilitate notation, let us
introduce the following conventions: by $\mu[\cdot]$, we denote integration with
respect to the measure $\mu$. Moreover, set 
$\widetilde{\lambda}_i = \lambda_i - \mu[\lambda_i]$. For any subset 
$\{i_1, \ldots, i_d\} \subset \{1, \ldots, N\}$, write 
$\widetilde{\lambda}_{i_1 \ldots i_d} = \widetilde{\lambda}_{i_1} \cdots
\widetilde{\lambda}_{i_d}$. Now (similarly to Theorem 1.4 in \cite{G-S-S}) define
the functions
\begin{align*}
f_1(\lambda) & = \sum_{i=1}^{N} a_i \widetilde{\lambda}_i,\\
f_2(\lambda) & = \sum_{i\ne j} a_{ij} (\widetilde{\lambda}_{ij} -
\mu[\widetilde{\lambda}_{ij}]),\pagebreak[2]\\
f_3(\lambda) & = \sum_{i\ne j\ne k} a_{ijk} \big(\widetilde{\lambda}_{ijk} -
\mu[\widetilde{\lambda}_{ijk}] - 3 \widetilde{\lambda}_i
\mu[\widetilde{\lambda}_{jk}]\big),\\
f_4(\lambda) & = \sum_{i\ne j\ne k\ne l} a_{ijkl} \big(\widetilde{\lambda}_{ijkl} -
\mu[\widetilde{\lambda}_{ijkl}] - 4 \widetilde{\lambda}_{i}
\mu[\widetilde{\lambda}_{jkl}] - 6 \widetilde{\lambda}_{ij}
\mu[\widetilde{\lambda}_{kl}] + 6\mu[\widetilde{\lambda}_{ij}]
\mu[\widetilde{\lambda}_{kl}]\big).
\end{align*}
Applying Theorem \ref{kontinuierlichmAbl} and recalling that the Sobolev
constant of $\mu$ is of order $1/N$ immediately yields the following result.

\vskip5mm
\begin{proposition}
\label{multPolEigen}
Let $\mu$ be the joint distribution of the ordered eigenvalues of $\Xi$ or the
distribution defined in \eqref{beta-ensemble}. For the functions $f_d$, 
$d = 1,\ldots, 4$, defined above, with some constant $c>0$, we have
$$
\int \exp \Big\{\frac{cN}{\lVert A \rVert_{\mathrm{HS}}^{2/d}}\, \lvert f_d
\rvert^{2/d} \Big\}\, d\mu \, \le \, 2
$$
and moreover
$$
\int \exp \Big\{\frac{c}{\lVert A \rVert_{\infty}^{2/d}}\,
\lvert f_d \rvert^{2/d} \Big\}\, d\mu \, \le \, 2.
$$
If $\mu$ is the eigenvalue distribution of $\Xi$, $c$ depends on the
logarithmic Sobolev constant $\sigma^2$, and if $\mu$ is the $\beta$-ensemble 
distribution \eqref{beta-ensemble}, $c$ depends on $\beta$ and the potential 
function $V$. In particular,
$$
\mu(\lvert f_d \lvert \ge t) \, \le \, 2 \exp\Big\{- \frac{cNt^{2/d}}
{\lVert A \rVert_{\mathrm{HS}}^{2/d}}\Big\} \, \le \,
2 \exp\Big\{- \frac{ct^{2/d}} {\lVert A \rVert_{\infty}^{2/d}}\Big\}.
$$
\end{proposition}

\vskip2mm
The bounds may be somewhat sharpened by applying Corollary \ref{KorrTailskont}. We
omit details.
In particular, Proposition \ref{multPolEigen} implies that if we ``recenter''
$$\sum_{i_1 \ne \ldots \ne i_d} \lambda_{i_1} \cdots \lambda_{i_d}\qquad (d = 1,
\ldots, 4)$$
in such a way that all derivatives of order up to $d-1$ are centered (cf. the
definition of the functions $f_d$ given above), we obtain exponential concentration
bounds which yield fluctuations of order $\mathcal{O}_P(1)$. For $d = 2$, we thus
get back a result shown in Proposition 1.12 from \cite{G-S}. These bounds may be
extended to higher orders $d \ge 5$. In some sense, this may be seen as an extension
of the self-normalizing property of linear eigenvalue statistics for a special class
of higher order polynomials.

\vskip5mm
\subsection{Functions on the unit sphere}

As a particular case, one may consider (real-valued) functions defined 
in some open neighbourhood $G$ of the unit sphere
$$
S^{n-1} = \{x \in \mathbb{R}^n \colon |x| = 1\},\qquad n \ge 2,
$$
which we equip with the uniform, or normalized Lebesgue measure
$\sigma_{n-1}$. Since any $\mathcal{C}^d$-smooth function
on $S^{n-1}$ can be extended to a $\mathcal{C}^d$-smooth function on
$\mathbb{R}^n \setminus \{0\}$, this means no loss of generality.
We then restrict the usual (Euclidean) derivatives of $f$
to $S^{n-1}$, which allows
to use the definitions of the hyper-matrices
\eqref{Hesseallgem} with operator norms \eqref{Operatornorm}, together
with the $L^p$-norms $\lVert f^{(d)} \rVert_{\mathrm{Op}, p}$ in
\eqref{Lpnorm2} taken with respect to $\sigma_{n-1}$.
This yields the following analogue of Theorem \ref{kontinuierlich}.

\vskip5mm
\begin{satz}
\label{sphere}
Let $f$ be a $\mathcal{C}^d$-smooth function on some open neighbourhood 
of $S^{n-1}$ with $\int_{S^{n-1}} f d\sigma_{n-1} = 0$. Assume that
$$
\lVert f^{(k)} \rVert_{\mathrm{Op},2} \, \le \, n^{-(d-k)/2} \qquad 
\forall\, k = 1, \ldots, d-1
$$
and
$\lvert f^{(d)}(\theta) \rvert_\mathrm{Op} \, \le \, 1$ for all 
$\theta \in S^{n-1}$. Then, for some universal constant $c > 0$,
$$
\int_{S^{n-1}} \exp \{(n-1)\, |f|^{2/d}/(8e)\}\, d\sigma_{n-1} 
\, \le \, 2.
$$
\end{satz}

\vskip2mm
Moreover, an analogue of Theorem \ref{kontinuierlichmAbl} also holds, which
is particularly interesting for the class of $p$-homogeneous functions.
Recall that a function $f$ on $\mathbb{R}^n \setminus \{0\}$
is $p$-homogeneous for some $p \in \mathbb{R}$, if
$f(\lambda x) = \lambda^p f(x)$ for all $x \neq 0$ and $\lambda > 0$.

\vskip5mm
\begin{satz}
\label{spheremAbl}
Suppose that a $\mathcal{C}^d$-smooth function $f$ on 
$\mathbb{R}^n \setminus \{0\}$ is $p$-homogeneous for some
real number $p > d-3$ and is orthogonal in $L^2(\sigma_{n-1})$ to all polynomials 
of total degree at most $d-1$. Moreover, assume that
\begin{equation}\label{Condsph}
\lVert f^{(d)} \rVert_{\mathrm{HS}, 2} \le 
1\qquad \text{and}\qquad \lVert f^{(d)} \rVert_{\mathrm{Op}, \infty} \le 1.
\end{equation}
Then, with some universal constant $c > 0$
$$
\int_G \exp \Big\{\frac{c}{\sigma^2} |f|^{2/d}\Big\}\, d\mu \, \le \, 2.
$$
A possible choice is $c = 1/(8 e)$.
The same holds for $p \le d-3$, if $n > d-p-1$.
\end{satz}

\vskip2mm
Recall that the Hilbert space $L^2(S^{n-1})$ can be decomposed into a sum 
of orthogonal subspaces $H_d, d = 0, 1, 2, \ldots$, consisting of all
$d$-homogeneous harmonic polynomials (in fact, restrictions of such polynomials 
to the sphere). This fact is mirrored in the orthogonality
assumptions from Theorem \ref{spheremAbl}.
If $f$ is not a homogeneous function, the bounds from Theorem \ref{spheremAbl} 
remain valid assuming \eqref{Condsph}, but instead of orthogonality
to polynomials of lower degree we have to require that $f$ and all its 
partial derivatives 
of order up to $d-1$ are centered with respect to~$\sigma_{n-1}$.

In Theorem \ref{sphere}, we use the usual (Euclidean) derivatives 
of functions defined in an open neighbourhood of the unit sphere.
In applications, this is usually sufficient. There is also 
a notion of intrinsic (spherical) derivatives (cf. Section 5), and it
is possible to obtain an analogue of Theorem \ref{sphere} for these derivatives as well.

To fix some notation, denote by $\nabla_S f$ the spherical gradient of 
a differentiable function $f \colon S^{n-1} \to \mathbb{R}$ and write 
$D_i f = \langle \nabla_S f, e_i \rangle$, $i = 1, \ldots, n$, 
for the spherical partial derivatives of $f$. Here, $e_i$ denotes 
the $i$-th standard unit vector in $\mathbb{R}^n$.
Higher order spherical partial derivatives are defined by iteration, e.\,g. 
$D_{ij} f = \langle \nabla_S \langle \nabla_S f, e_j \rangle, e_i \rangle$
for any $1 \le i,j \le n$. Note that in general, $D_{ij}f \ne D_{ji} f$. 
If $f \in \mathcal{C}^d(S^{n-1})$, we denote by $D^d f (\theta)$ the hyper-matrix
of the spherical partial derivatives of order $d$, i.\,e.
\begin{equation}
\label{Hessesph}
(D^{(d)} f(\theta))_{i_1 \ldots i_d} = D_{i_1 \ldots i_d} f(\theta), \qquad
\theta \in S^{n-1}.
\end{equation}
Similarly to \eqref{Operatornorm}, let $|D^{(d)} f (\theta)|_\mathrm{Op}$ 
be the operator norm of $D^{(d)}f (\theta)$. Finally, write
\begin{equation}
\label{Lpnormsph}
\lVert D^{(d)} f \rVert_{\mathrm{Op}, p} \, = \,
\Big( \int_{S^{n-1}} |D^{(d)}f|_\mathrm{Op}^p\, d\sigma_{n-1}\Big)^{1/p},
\qquad p \in (0, \infty].
\end{equation}

We have the following ``intrinsic'' version of Theorem \ref{sphere}.

\vskip5mm
\begin{satz}
\label{sphere-intr}
Let $f$ be a $\mathcal{C}^d$-smooth function on 
$S^{n-1}$ such that $\int_{S^{n-1}} f\, d\sigma_{n-1} = 0$. Assume that
$$
\lVert D^{(k)} f \rVert_{\mathrm{Op},2} \le n^{-(d-k)/2}, \quad 
k = 1, \ldots, d-1,
$$
and
$\lvert D^{(d)} f(\theta) \rvert_\mathrm{Op} \le 1$ for all 
$\theta \in S^{n-1}$.
Then
$$
\int_{S^{n-1}} \exp \{(n-1)\, |f|^{2/d}/(8e)\}\, d\sigma_{n-1} \, \le \, 2.
$$
\end{satz}

\vskip5mm
\subsection{Outline}

In Section 2, we give the proofs of the theorems and corollaries from Section 1.1.
We briefly discuss the notion of difference operators. The main tool is 
a recursion inequality for the $L^p$-norms of the function $f$ and 
the Hilbert--Schmidt norms of $|\mathfrak{h}^{(k)}f|$.
In Section 3, Theorem \ref{HigherOrderEfronStein} is proven. 
This includes a number of relations between the difference operators introduced
in Definition \ref{DiffOp}.
In Section 4, we prove Theorems \ref{kontinuierlich} and 
\ref{kontinuierlichmAbl} as well as Corollary \ref{KorrTailskont}
by adapting the main steps of the proof of 
Theorem \ref{allgem}. In Section 5, the proofs of 
Theorems \ref{sphere}, \ref{spheremAbl} and \ref{sphere-intr} are given;
in particular, we introduce some facts about spherical calculus which allow us 
to proceed in a similar way as in case of functions on $\mathbb{R}^n$.
Finally, in Section 6, we illustrate Theorem \ref{sphere} on the
example of polynomials and the problem of Edgeworth approximations for
symmetric functions on the sphere.
For additional applications we refer to \cite{G-S}.

\vskip5mm
\section{Functions of independent random variables: Proofs}

Let $X = (X_1, \ldots, X_n)$ be a vector of independent random variables 
on the probability space $(\Omega, \mathcal{A}, \linebreak[3]\mathbb{P})$. 
By a ``difference operator'' we mean an $\mathbb{R}^n$-valued
functional $\Gamma$ defined on $L^\infty(\mathbb{P})$
such that the following two conditions hold:

\vskip5mm
\begin{condition}

\begin{enumerate} [(i)]
\item 
$\Gamma f(X) = (\Gamma_1 f(X), \ldots \Gamma_n f(X))$, where 
$f:\mathbb{R}^n \rightarrow \mathbb{R}$ may be any Borel measurable function 
such that $f(X) \in L^\infty(\mathbb{P})$.

\item 
$|\Gamma_i (af(X) + b)| = a\, |\Gamma_i f(X)|$ for all 
$a > 0$, $b \in \mathbb{R}$ and $i = 1, \ldots, n$.
\end{enumerate}
\label{DiskrDiffallgem}
\end{condition}

\vskip2mm
We also call $\Gamma$ a gradient operator or simply gradient.
We do not suppose $\Gamma$ to satisfy any sort of ``Leibniz rule''. 
Clearly, the difference operator $\mathfrak{h}$ from \eqref{h} and any of the
difference operators introduced in Definition \ref{DiffOp} satisfy 
Conditions \ref{DiskrDiffallgem}.

For the proof of Theorem \ref{allgem} we will need several lemmas. 
As before, let
$T_i f = f(X_1, \ldots, X_{i-1}, \bar{X}_i,\linebreak[2] X_{i+1}, \ldots, X_n)$
with $\bar{X}_1, \ldots, \bar{X}_n$ an independent copy of $X$.
As a first step, the Hilbert--Schmidt norms of 
the derivatives of consecutive orders are related in the following way:

\vskip5mm
\begin{lemma}
\label{itGradHessindep}
For any $d \ge 2$,
\begin{equation}
\label{Rekursion}
|\mathfrak{h}|\mathfrak{h}^{(d-1)}f(X)|_\mathrm{HS}| \le 
|\mathfrak{h}^{(d)}f(X)|_\mathrm{HS}.
\end{equation}
\end{lemma}

\vskip2mm
\begin{proof}
First let $d=2$. Using $T_i|\mathfrak{h}f| = |T_i\mathfrak{h}f|$ and 
the triangle inequality, we have
\begin{equation}
\label{Schritt1}
(\mathfrak{h}_i|\mathfrak{h}f|)^2 = 
\frac{1}{4}\, \lVert|\mathfrak{h}f| - |T_i\mathfrak{h}f|\rVert_{i,\infty}^2
 \le 
\frac{1}{4}\, \lVert\lvert\mathfrak{h}f - T_i\mathfrak{h}f\rvert\rVert_{i,\infty}^2
  = 
\frac{1}{4}\, \lVert\lvert\mathfrak{h}f - T_i\mathfrak{h}f\rvert^2\rVert_{i,\infty}.
\end{equation}
Here, $\mathfrak{h}f - T_i\mathfrak{h}f$ is defined componentwise. Since 
$T_i \mathfrak{h}_i f = \mathfrak{h}_i f$, we obtain
\begin{align}
|\mathfrak{h}f - T_i\mathfrak{h}f|^2 
 & = 
\sum_{j \ne i}\, \frac{1}{4}\, \big(\lVert f - T_jf \rVert_{j,\infty}
 - \lVert T_i f - T_{ij}f \rVert_{j,\infty} \big)^2\notag\\
 & \le 
\sum_{j \ne i}\, \frac{1}{4}\, 
\lVert f - T_j f - T_i f + T_{ij} f \rVert_{j,\infty}^2,\label{Schritt2}
\end{align}
where the last inequality follows from the reverse triangular inequality again 
(for the pseudo-norm $\lVert \cdot \rVert_{j, \infty}$).
Combining \eqref{Schritt1} and \eqref{Schritt2} yields
\begin{align*}
(\mathfrak{h}_i|\mathfrak{h}f|)^2 
 & \le 
\big\lVert \frac{1}{16}\, \sum_{j \ne i}\,
\lVert f - T_j f - T_i f + T_{ij} f \rVert_{j,\infty}^2 \big\rVert_{i,\infty}\\
 & \le 
\frac{1}{16}\, \sum_{j \ne i}\,
\lVert f - T_j f - T_i f + T_{ij} f \rVert_{i,j,\infty}^2.
\end{align*}
Summing over $i = 1, \ldots, n$ we arrive at the result in the case $d=2$.

For $d \ge 3$, note that
$T_i \mathfrak{h}_{i_1 \ldots i_{d-1}} f = \mathfrak{h}_{i_1 \ldots i_{d-1}} f$
whenever $i \in \{i_1, \ldots, i_{d-1}\}$.
The claim then follows in the same way as above.
\end{proof}

\vskip2mm
Corresponding results in the setting of differentiable functions 
(see Lemma \ref{itGradHess}) suggest to replace the Hilbert--Schmidt norms in 
Lemma \ref{itGradHessindep} by operator type norms \eqref{Operatornorm}. 
In Boucheron, Bousquet, Lugosi and Massart \cite{B-B-L-M}, Theorem 14, 
iterations of \eqref{BBLM1} are sketched to study 
applications for Rademacher chaos type functions. 
Unfortunately, working out the arguments 
in the proof of Theorem 14 we seemed to need Hilbert--Schmidt instead of 
operator norms. Already in second order statistics of Rademacher variables 
like $\sum_{i=1}^{n} X_iX_{i+1}$ (setting $X_{n+1} = X_1$), an analogue of 
\eqref{Rekursion} for operator norms cannot be true. Similar remarks hold
for any of the difference operators introduced in Definition \ref{DiffOp}
(cf. Remark \ref{DiskrDiff}).

Our results will follow from certain moment inequalities for functions 
of independent random variables. In \cite{B-B-L-M}, cf. Theorem 2,
 the following moment bounds are shown
$$
\lVert (f - \mathbb{E}f)_+ \rVert_p \le 
\sqrt{2\kappa p}\, \lVert V^+(f) \rVert_p, \qquad
\lVert (f - \mathbb{E}f)_- \rVert_p \le 
\sqrt{2\kappa p}\, \lVert V^-(f) \rVert_p
$$
in terms of the conditional expectations 
$$
V^+(f) = \mathbb{E}\, \Big( \sum_{i=1}^{n} (f - T_if)_+^2 \big| X \Big) \qquad
V^-(f) = \mathbb{E}\, \Big( \sum_{i=1}^{n} (f - T_if)_-^2 \big| X \Big),
$$
where $\kappa = \frac{\sqrt{e}}{2\,(\sqrt{e} - 1)} < 1.271$. 
Note that, in our notations according to Definition \ref{DiffOp}, 
$$
V^+(f) = 2\, |\mathfrak{d}^+ f|^2 \qquad {\rm and} \qquad 
V^-(f) = 2\, |\mathfrak{d}^- f|^2.
$$ 
For iterating these inequalities however, we had to bypass the problem that
$\mathfrak{d}_{ii} = \mathfrak{d}_i$ respectively 
$\mathfrak{d}^+_{ii} = \mathfrak{d}^+_i$ (up to constant), which would introduce 
additional lower order differences on the right-hand side of \eqref{Rekursion}. 
This motivated us to introduce the following related quantities adapted to 
the framework of $L^\infty$-bounds. For $i = 1, \ldots, n$, introduce
\begin{equation}
\label{h+}
\mathfrak{h}^+_if(X) = \frac{1}{2}\, \lVert (f(X) - T_if(X))_+ \rVert_{i, \infty},
\qquad \mathfrak{h}^+f = (\mathfrak{h}^+_1f, \ldots, \mathfrak{h}^+_nf),
\end{equation}
\begin{equation}
\label{h-}
\mathfrak{h}^-_if(X) = \frac{1}{2}\, \lVert (f(X) - T_if(X))_- \rVert_{i, \infty},
\qquad \mathfrak{h}^-f = (\mathfrak{h}^-_1f, \ldots, \mathfrak{h}^-_nf),
\end{equation}
which are clearly difference operators in the sense of 
Conditions \ref{DiskrDiffallgem}. Using the relations
$V^+(f) \le 4\, |\mathfrak{h}^+f|^2$ and $V^-(f) \le 4\, |\mathfrak{h}^-f|^2$,
we get from the \cite{B-B-L-M}-result the following somewhat weaker bounds
in terms of the $L^p$-norms as in \eqref{LpNorm1}.

\vskip5mm
\begin{satz}
\label{BBLM}
With the same constant $\kappa = \frac{\sqrt{e}}{2\,(\sqrt{e} - 1)}$,
for any real $p \ge 2$,
\begin{equation}
\label{BBLM1}
\lVert (f - \mathbb{E}f)_+ \rVert_p \le 
\sqrt{8\,\kappa p}\, \lVert \mathfrak{h}^+f \rVert_p,
\end{equation}
\begin{equation}
\label{BBLM2}
\lVert (f - \mathbb{E}f)_- \rVert_p \le \sqrt{8\,\kappa p}\, 
\lVert \mathfrak{h}^-f \rVert_p.
\end{equation}
\end{satz}

\vskip2mm
For the reader's convenience, let us give a self-contained proof of 
Theorem \ref{BBLM}. It is sufficient to derive \eqref{BBLM1}, 
since \eqref{BBLM2} follows from \eqref{BBLM1} by considering $-f$. 
The key step are the following two lemmas.

\vskip5mm
\begin{lemma}
Assume $\mathbb{E} f = 0$. Then,
\begin{equation}
\label{Le1}
\lVert f \rVert_2 \le \sqrt{2}\, \lVert\mathfrak{h}f\rVert_2,
\end{equation}
\begin{equation}
\label{Le2}
\lVert f_+ \rVert_2 \le 2\, \lVert\mathfrak{h}^+f\rVert_2.
\end{equation}
\end{lemma}

\vskip2mm
\begin{proof}
By the Efron--Stein inequality \eqref{Efron-Stein},
$\mathbb{E}f^2 \le \mathbb{E}\,|\mathfrak{d}f|^2$ and
$\mathbb{E}f^2 \le 2\, \mathbb{E}\,|\mathfrak{d}^+f|^2$, while
$|\mathfrak{d}^+ f|^2 \le 2\,|\mathfrak{h}^+ f|^2$ and 
$|\mathfrak{d} f|^2 \le 2\,|\mathfrak{h} f|^2$ (cf. Remark \ref{DiskrDiff} $(v)$).
\end{proof}

\vskip2mm
The next lemma provides a moment recursion similarly to 
\cite{B-B-L-M}, {Lemma 3}.

\vskip5mm
\begin{lemma}
For any real $p \ge 2$,
\begin{equation}
\label{Rek}
\lVert f_+ \rVert_p^p \, \le \, 
\lVert f_+ \rVert_{p-1}^p + 4\, (p-1)\, \lVert \mathfrak{h}^+ f \rVert_p^2\, 
\lVert f_+ \rVert^{p-2}_p.
\end{equation}
\end{lemma}

\vskip2mm
\begin{proof}
First assume $n=1$, i.\,e. $f = f(X)$ for a random variable $X$ and 
$Tf = f(\bar{X})$, where $\bar{X}$ is an independent copy of $X$.
Using the notation $f_+^{p-1} \equiv (f_+)^{p-1}$, we have
$$
\frac{1}{2}\, \mathbb{E}\, 
[(f_+^{p-1} - Tf_+^{p-1})(f_+ - Tf_+)] - \lVert f_+ \rVert_p^p
  \, = \,
- \lVert f_+ \rVert_{p-1}^{p-1}\, \lVert f_+ \rVert_1 
  \, \ge \, 
- \lVert f_+ \rVert_{p-1}^p.
$$
Thus, by symmetry in $X$ and $\bar{X}$, and since
$(f - T f)_+ \le 2\,\mathfrak{h}^+ f$,
\begin{align*}
\lVert f_+ \rVert_p^p 
  & \le 
\lVert f_+ \rVert_{p-1}^p + \mathbb{E}\, [(f_+^{p-1} - Tf_+^{p-1})_+\,(f_+ - Tf_+)] \\
  & \leq
\lVert f_+ \rVert_{p-1}^p + (p-1)\, \mathbb{E}\, \big[(f-Tf)_+^2 \,f_+^{p-2}\big].	
\end{align*}
Using H\"older's inequality, the last expectation may be bounded by
\begin{align}
\label{n=1}
\begin{split}
4\,\mathbb{E}\, \big[\,|\mathfrak{h}^+ f|^2 f_+^{p-2}\big]
 \, \le \,
4\,\lVert |\, \mathfrak{h}^+ f|^2 \rVert_{p/2}\, \lVert f_+ \rVert^{p-2}_p
 \,  = \,
4\,\lVert \mathfrak{h}^+f\, \rVert_p^2\ \lVert f_+ \rVert^{p-2}_p.
\end{split}
\end{align}
This completes the proof in case $n=1$.

For $n \ge 1$, we use a tensorization argument:
For any $g \in L^q$, $q \in (1,2]$,
\begin{equation}
\label{Tensorisierung}
\mathbb{E}\, |g|^q - \big(\mathbb{E}\, |g|\big)^q \le 
\mathbb{E}\, \sum_{i=1}^{n} \Big(\mathbb{E}_i\, |g|^q - 
\big(\mathbb{E}_i\, |g|\big)^q\Big),
\end{equation}
where $\mathbb{E}_i$ denotes expectation with respect to $X_i$. 
Applying this inequality to $g = f_+^{p-1}$ with $q = p/(p-1)$,
similarly to the case of $n=1$ we obtain
\begin{align*}
\lVert f_+ \rVert_p^p - \lVert f_+ \rVert_{p-1}^p 
 & \, \le \,
\mathbb{E} \ \sum_{i=1}^{n} 
\Big( \mathbb{E}_i f_+^p - (\mathbb{E}_i f_+^{p-1})^{p/(p-1)} \Big)\\
 & \, \le \,
(p-1)\, \sum_{i=1}^{n}\, \mathbb{E}\ \mathbb{E}_i\, \bar{\mathbb{E}}_i\, 
\big[(f - T_i f)^2_+ f_+^{p-2}\big] \\
 & \, \le \,
4\,(p-1)\, \sum_{i=1}^{n}\, \mathbb{E}\, \big[|\mathfrak{h}_i^+ f|^2 f_+^{p-2}\big]
 \, = \,
4\,(p-1)\, \mathbb{E}\, \big[|\mathfrak{h}^+ f|^2 f_+^{p-2}\big].
\end{align*}
As in \eqref{n=1}, the last expectation is bounded by
$\lVert \mathfrak{h}^+f \rVert_p^2\, \lVert f_+ \rVert^{p-2}_p$
using H\"older's inequality, which gives the desired result.

It remains to prove \eqref{Tensorisierung}. Let us mention that
the tensorization of functionals
$L(g) = \mathbb{E}\,\Psi(g) - \Psi(\mathbb{E}\,g)$ was proposed in the mid
1990's by Bobkov, as explained in \cite{L1}, Proposition 4.1. This property
is actually equivalent to the convexity of $L$ in $g$, and can be explicitly
expressed in terms of $R$ (convexity of $\Psi$ and $-1/\Psi''$; see also
\cite{L-O}). 
For completeness of  exposition let us include here a direct argument
for the 	 
power functions $\Psi(x) = x^q$. By induction, it suffices 
to consider $n=2$; we use the representation
\begin{equation}
\label{Variational}
L(|g|)  = \sup_{h \in L^q} 
\big\{q\, \mathbb{E}\,\big[|g|(|h|^{q-1} - (\mathbb{E}\,|h|)^{q-1})\big] - 
(q-1)\, \big(\mathbb{E}\, |h|^q - (\mathbb{E}\, |h|)^q\big) \big\}.
\end{equation}
Indeed, by the arithmetic-geometric inequality, 
$
\frac{1}{q}\, \mathbb{E}\, |g|^q + \frac{q-1}{q}\, \mathbb{E}\, |h|^q \le 
\mathbb{E}\, |g|\, |h|^{q-1},
$
which we rewrite as
$$
\mathbb{E}\, |g|^q \, \le \,
q\, \mathbb{E}\, |g|\, |h|^{q-1} - (q-1)\, \mathbb{E}\, |h|^q.
$$
We may assume $\mathbb{E}|g| = 1$; therefore, subtracting 
$(\mathbb{E}\, |g|)^q = 1$ on both sides,
$$
L(|g|) \le q\, \mathbb{E}\, 
\big[|g|\, (|h|^{q-1} - (\mathbb{E}\, |h|)^{q-1})\big] - 
(q-1)\, \big(\mathbb{E}\, |h|^q
- (\mathbb{E}\, |h|)^q\big) + R(\mathbb{E}\, |h|)
$$
with $R(x) = q x^{q-1} - (q-1)\,x^q - 1$ for $x \ge 0$. Since
$R(x) \le 0$, while equality holds if $h = g$, we arrive at 
\eqref{Variational}.
By Fubini's theorem and applying \eqref{Variational}, we now get
\begin{align*}
 &
\mathbb{E}_2 \big[(\mathbb{E}_1 |g|)^q - (\mathbb{E} |g|)^q\big] 
 = 
\mathbb{E}_2 (\mathbb{E}_1 |g|)^q - (\mathbb{E}_2 \mathbb{E}_1 |g|)^q \\
 = \; &
\sup_{h(X_2) \in L^q} \big\{q\, \mathbb{E}_2\big[(\mathbb{E}_1|g|)(|h|^{q-1} - 
(\mathbb{E}_2|h|)^{q-1})\big] - 
(q-1) \big(\mathbb{E}_2 |h|^q - (\mathbb{E}_2 |h|)^q\big) \big\}\pagebreak[2]\\
 = \; &
\sup_{h(X_2) \in L^q} \big\{\mathbb{E}_1
\big[q\, (\mathbb{E}_2|g|)(|h|^{q-1} - (\mathbb{E}_2|h|)^{q-1}) - 
(q-1) \big(\mathbb{E}_2 |h|^q - (\mathbb{E}_2 |h|)^q\big)\big] \big\}\pagebreak[2]\\
 \le \; &
\mathbb{E}_1\big[\sup_{h(X_2) \in L^q} 
\big\{q\, \mathbb{E}_2\big[|g|(|h|^{q-1} - (\mathbb{E}_2|h|)^{q-1})\big] - 
(q-1) \big(\mathbb{E}_2 |h|^q - (\mathbb{E}_2 |h|)^q\big) \big\}\big]\\
 = \; &\mathbb{E}_1 \big[\mathbb{E}_2 |g|^q - (\mathbb{E}_2 |g|)^q\big].
\end{align*}
As a consequence, by Fubini's theorem again,
\begin{align*}
\mathbb{E}\, |g|^q - (\mathbb{E}\, |g|)^q
 & = 
\mathbb{E}_2 \big[\mathbb{E}_1|g|^q - (\mathbb{E}_1 |g|)^q\big] + \mathbb{E}_2 
\big[(\mathbb{E}_1|g|)^q - (\mathbb{E} |g|)^q\big]\\
 & \le 
\mathbb{E}_2 \big[\mathbb{E}_1|g|^q - (\mathbb{E}_1 |g|)^q\big] + 
\mathbb{E}_1 \big[\mathbb{E}_2 |g|^q - (\mathbb{E}_2 |g|)^q\big].
\end{align*}
\end{proof}

\vskip2mm
Following the arguments in \cite{B-B-L-M}, we may now prove 
Theorem \ref{BBLM}:

\vskip5mm
\begin{proof}[Proof of Theorem \ref{BBLM}]
It suffices to prove \eqref{BBLM1} assuming $\mathbb{E} f = 0$. 
To this end, by induction on $k$, we show that for all $k \in \mathbb{N}$ 
and all $p \in (k,k+1]$,

\begin{equation}
\label{Induction}
\lVert f_+ \rVert_p \le 
\sqrt{8\kappa_p\, p}\, \lVert \mathfrak{h}^+ f \rVert_{p \vee 2} \quad 
{\rm with} \ \
\kappa_p = \frac{1}{2}\, \Big(1 - \Big(1 - \frac{1}{p}\Big)^{p/2}\Big)^{-1}.
\end{equation}
These constants are strictly increasing in $p$, $\kappa_1 = 1/2$ and 
$\lim_{p \to \infty} \kappa_p = \kappa = \frac{\sqrt{e}}{2\,(\sqrt{e} - 1)}$.

For $k = 1$ and $p \in (1,2]$, by \eqref{Le2} and the fact that 
$\kappa_p p \ge 1/2$, we have
$$
\lVert f_+ \rVert_p \le \lVert f_+ \rVert_2 \le 
2\, \lVert\mathfrak{h}^+f\rVert_2 \le 
\sqrt{8\kappa_p\, p}\, \lVert \mathfrak{h}^+ f \rVert_2.
$$

To make an induction step, fix an integer $k > 1$ and assume that
\eqref{Induction} holds for all real $p \in [1,k]$. Now,
consider the values $p \in (k, k+1]$. Set
$$
x_p = \lVert f_+ \rVert_p^p\ 8^{-p/2} \kappa_p^{-p/2} p^{-p/2} 
\lVert \mathfrak{h}^+ f\, \rVert_{p \vee 2}^{-p},
$$
so that it suffices to prove that $x_p \le 1$. In terms of $x_p$, \eqref{Rek} 
implies that
\begin{align*}
       & 
x_p\ 8^{p/2} \kappa_p^{p/2} p^{p/2}\, \lVert \mathfrak{h}^+ f \rVert_p^p\\
\le \, & 
x_{p-1}^{p/(p-1)}\ 8^{p/2} \kappa_{p-1}^{p/2}\, (p-1)^{p/2}\, 
\lVert \mathfrak{h}^+ f \rVert_{p-1}^p\\
       & + 
4\, (p-1)\, \lVert \mathfrak{h}^+ f \rVert_p^2\ x_p^{1-2/p}\ 8^{p/2-1} 
\kappa_p^{p/2 - 1} p^{p/2 - 1}\, \lVert \mathfrak{h}^+ f\, \rVert_p^{p-2}\\
 	\le \, &
x_{p-1}^{p/(p-1)}\ 8^{p/2} \kappa_p^{p/2} (p-1)^{p/2}\, 
\lVert \mathfrak{h}^+ f \rVert_p^p + \frac{1}{2}\, x_p^{1-2/p}\ 
8^{p/2} \kappa_p^{p/2 - 1}\, p^{p/2}\, \lVert \mathfrak{h}^+ f \rVert_p^p.
\end{align*}
Here we have used the fact that $\kappa_{p-1} \le \kappa_p$. 
Simplifying and using that by induction, $x_{p-1} \le 1$, it follows that
$$
x_p \le x_{p-1}^{p/(p-1)}\Big(1 - \frac{1}{p}\Big)^{p/2} + 
\frac{1}{2 \kappa_p}\, x_p^{1-2/p} \le \Big(1 - \frac{1}{p}\Big)^{p/2} + 
\frac{1}{2 \kappa_p}\, x_p^{1-2/p}.
$$
Now note that the function
$$
u_p(x) =
\Big(1 - \frac{1}{p}\Big)^{p/2} + \frac{1}{2 \kappa_p}\, x^{1-2/p} - x
$$
is concave on $\mathbb{R}_+$ and positive at $x = 0$. Since $u_p(1) = 0$ 
and $u_p(x_p) \ge 0$, we may conclude that $x_p \le 1$.
\end{proof}

\vskip5mm
\begin{korollar}
\label{momentineq}
Given
$f = f(X_1, \ldots,\linebreak[3] X_n)$ in $L^\infty(\mathbb{P})$, 
for all $p \ge 2$,
\begin{equation}
\label{Ineq1}
\lVert f \rVert_p \le \lVert f \rVert_2 + \sqrt{32\, \kappa p}\, 
\lVert \mathfrak{h}f \rVert_p.
\end{equation}
If additionally $\mathbb{E} f = 0$,
\begin{equation}
\label{Ineq2}
\lVert f \rVert_p \le \sqrt{32\, \kappa p}\, \lVert \mathfrak{h}f \rVert_p.
\end{equation}
\end{korollar}

\vskip2mm
\begin{proof}
By Theorem \ref{BBLM},
\begin{align*}
\lVert f - \mathbb{E}f \rVert_p 
 & \le 
\lVert (f - \mathbb{E}f)_+ \rVert_p + \lVert (f - \mathbb{E}f)_- \rVert_p\\
 & \le 
\sqrt{8 \kappa p}\, \lVert \mathfrak{h}^+f \rVert_p + \sqrt{8 \kappa p}\, 
\lVert \mathfrak{h}^-f \rVert_p
 \ \le \
2 \sqrt{8 \kappa p}\, \lVert \mathfrak{h}f \rVert_p,
\end{align*}
which proves \eqref{Ineq2}. Moreover, by the triangle inequality,
\begin{align*}
\lVert f - \mathbb{E}f \rVert_p \, \ge \,
\lVert f \rVert_p - \lvert\mathbb{E} f\rvert \, \ge \,
\lVert f \rVert_p - \lVert f \rVert_2,
\end{align*}
so that we obtain \eqref{Ineq1}.
\end{proof}

We shall now prove Theorem \ref{allgem}. Recall that if the relation 
of the form
\begin{equation}
\label{subexp}
\lVert f \rVert_k \le \gamma k \qquad (k \in \mathbb{N})
\end{equation}
holds true with some constant $\gamma > 0$, then $f$ has sub-exponential 
tails, i.\,e. $\mathbb{E} e^{c|f|} \le 2$ for some constant $c = c(\gamma) > 0$,
e.\,g. $c = \frac{1}{2 \gamma e}$. Indeed, using $k! \ge (\frac{k}{e})^k$,
we have
$$
\mathbb{E} \exp(c|f|) = 1 + \sum_{k=1}^{\infty} c^k \frac{\mathbb{E}\, |f|^k}{k!} 
\le 1 + \sum_{k=1}^{\infty} (c \gamma)^k \frac{k^k}{k!}
\le 1 + \sum_{k=1}^{\infty} (c \gamma e)^k = 2.
$$

\vskip2mm
\begin{proof}[Proof of Theorem \ref{allgem}]
Put $A = \sqrt{32\,\kappa p}$.
Using \eqref{Ineq1} with $f$ replaced by 
$\lvert \mathfrak{h}^{(k-1)} f \rvert_\mathrm{HS}$ for $k = 1, \ldots, d$, 
and applying Lemma \ref{itGradHessindep}, we get
\begin{align*}
\lVert \mathfrak{h}^{(k-1)}f \rVert_{\mathrm{HS}, p} 
 & \le 
\lVert \mathfrak{h}^{(k-1)}f \rVert_{\mathrm{HS}, 2} + A\, 
\lVert \mathfrak{h} \lvert \mathfrak{h}^{(k-1)}f \rvert_\mathrm{HS} \rVert_p\\
 & \le 
\lVert \mathfrak{h}^{(k-1)}f \rVert_{\mathrm{HS}, 2} + A\, 
\lVert \mathfrak{h}^{(k)}f \rVert_{\mathrm{HS}, p}.
\end{align*}
Consequently, using \eqref{Ineq2} and then iterating \eqref{Ineq1}, we arrive at
\begin{equation}
\lVert f \rVert_p
 \le 
\sum_{k=1}^{d-1} A^k \,
\lVert \mathfrak{h}^{(k)}f \rVert_{\mathrm{HS}, 2} + 
A^d \, \lVert \mathfrak{h}^{(d)}f 
\rVert_{\mathrm{HS},p}.\label{byit}
\end{equation}
Now, since 
$\lVert \mathfrak{h}^{(k)}f \rVert_{\mathrm{HS}, 2} \le 1$ for
$k \leq d-1$ and 
$\lVert \mathfrak{h}^{(d)}f \rVert_{\mathrm{HS}, \infty} \le 1$ by assumption, 
we get
$$
\lVert f \rVert_p \le \sum_{k=1}^{d} A^k =
\frac{A^{d+1} - 1}{A - 1} - 1\le 
\Big(\frac{A}{A-1} A\Big)^d
$$
for all $p \ge 2$, where $A/(A-1) \le 1.12$. We now arrive at $1.12A \le \sqrt{Cp}$,
where the best constant corresponds to $p=2$, and then we find that $C < 52$. Hence,
we obtain the bound 
\begin{equation}
\label{Momenteiteriertallgem}
\lVert f \rVert_p \le (52\, p)^{d/2}, \qquad p \geq 2.
\end{equation}
As for $0 < p < 2$, one may find
$\lVert f \rVert_p \le \lVert f \rVert_2 \le (104)^{d/2}$,
and thus, for all $k \geq 1$,
$$
\lVert \lvert f \rvert^{2/d} \rVert_k =
\lVert f \rVert_{2k/d}^{2/d}\le \gamma k,
$$
as in \eqref{subexp}, with constant $\gamma = 104$.
\end{proof}

\vskip2mm
\begin{proof}[Proof of Proposition \ref{multPol}]
First note that since $X_i$'s are centered, we have $\alpha_0 = 0$, and
the Hoeffding decomposition of $f$ is given by the polynomials
$$
h_{i_1 \ldots i_d}(X_{i_1}, \ldots, X_{i_d}) = 
\alpha_{i_1 \ldots i_d}\, X_{i_1} \cdots X_{i_d} \qquad 
(i_1 < \ldots < i_d, \ \ d = 1, \ldots, n).
$$
It is now easily seen that for any $1 \le k \le d$ and
$1 \le j_1 \ne \ldots \ne j_k \le n$,
\begin{align*}
\mathfrak{h}_{j_1 \ldots j_k} f(X) 
  = \; &
\mathfrak{h} X_{j_1} \cdots \mathfrak{h} X_{j_k} \Big| 
\sum_{\substack{i_1 < \ldots < i_d \\ \ni j_1, \ldots, j_k}}
\alpha_{i_1 \ldots i_d} \prod_{\substack{\nu \in \{i_1, \ldots, i_d\} \\ 
\setminus \{j_1, \ldots, j_k\}}} X_\nu\\
       & + 
\sum_{\substack{i_1 < \ldots < i_{d+1} \\ 
\ni j_1, \ldots, j_k}} \alpha_{i_1 \ldots i_{d+1}} 
\prod_{\substack{\nu \in \{i_1, \ldots, i_{d+1}\} \\
\setminus \{j_1, \ldots, j_k\}}} X_\nu + \ldots \Big|.
\end{align*}
Here, $\mathfrak{h} X_{j_\nu}$ is understood according to \eqref{h} as 
the difference $\mathfrak{h} g$ for the function $g(X_{j_\nu}) = X_{j_\nu}$ 
(i.\,e. in dimension $n = 1$). Hence, for any $k = 0, \ldots, d$,
\begin{align*}
\lVert \mathfrak{h}^{(k)} f(X) \rVert_{\mathrm{HS}, 2}^2 
  = \; &
\sum_{i_1 < \ldots < i_d} \alpha_{i_1 \ldots i_d}^2 
\sum_{\substack{j_1 \ne \ldots \ne j_k \\ 
\in \{i_1,\ldots, i_d\}}} (\mathfrak{h} X_{j_1})^2 \cdots (\mathfrak{h} X_{j_k})^2\\
       & + 
\sum_{i_1 < \ldots < i_{d+1}} \alpha_{i_1 \ldots i_{d+1}}^2 
\sum_{\substack{j_1 \ne \ldots \ne j_k \\ 
\in \{i_1, \ldots, i_{d+1}\}}} 
(\mathfrak{h} X_{j_1})^2 \cdots (\mathfrak{h} X_{j_k})^2 + \ldots
\end{align*}
As a consequence, since $\mathfrak{h} X_i \ge 1$ for all $i = 1, \ldots, n$,
\begin{equation}
\label{normen}
\lVert f \rVert_{\mathrm{HS}, 2} \le 
\lVert \mathfrak{h}^{(1)} f(X) \rVert_{\mathrm{HS}, 2} \le 
\ldots \leq \lVert \mathfrak{h}^{(d)} f(X) \rVert_{\mathrm{HS}, 2} \le 
\lVert \mathfrak{h}^{(d)} f(X) \rVert_{\mathrm{HS}, \infty}.
\end{equation}
\end{proof}

\vskip2mm
Corollary \ref{U-statistics} can now be obtained along the lines of the 
proof of Theorem \ref{allgem}, when the random variables $X_1, \ldots, X_n$ 
are Rademacher variables.

\vskip2mm
\begin{proof}[Proof of Corollary \ref{U-statistics}]

Without loss of generality we assume $h$ to be symmetric (i.\,e. invariant 
under permutations). Hence $f$ can be rewritten as
$$
f(X_1, \ldots, X_n) = \frac{1}{\binom{n}{d}} 
\sum_{i_1 < \ldots < i_d} h(X_{i_1}, \ldots, X_{i_d}).
$$
Introduce a set of independent Rademacher variables 
$\varepsilon_1, \ldots, \varepsilon_n$ which are independent of the random 
variables $X_1, \ldots, X_n$ and consider
$$
f^\varepsilon(X, \varepsilon) = 
\frac{1}{\binom{n}{d}} \sum_{i_1 < \ldots < i_d} 
\varepsilon_{i_1} \cdots \varepsilon_{i_d}\, h(X_{i_1}, \ldots, X_{i_d}).
$$
Denote by $\mathfrak{h}_i^\varepsilon$, $\mathfrak{h}^{\varepsilon(d)}$ 
and by similar expressions differences of $f^\varepsilon$ with
respect to the Rademacher variables $\varepsilon_i$ conditionally on $X$.

Note that conditionally on $X$, $f^\varepsilon(X, \varepsilon)$ has Fourier--Walsh 
expansion consisting of the $d$-th order term only. Hence, we may use
\eqref{normen} with $\lVert \mathfrak{h}^{(d)} f(X) \rVert_{\mathrm{HS}, \infty}$ 
replaced by $\lVert \mathfrak{h}^{(d)} f(X) \rVert_{\mathrm{HS}, p}$ in
\eqref{byit}. Arguing as in \eqref{Momenteiteriertallgem}, conditionally 
on $X$ we get
$$
\mathbb{E}_\varepsilon\, \lvert f^\varepsilon(X, \varepsilon) \rvert^p
 \le 
(52\, p)^{pd/2} \,
\mathbb{E}_\varepsilon\, 
\lvert \mathfrak{h}^{\varepsilon(d)}f^\varepsilon(X, \varepsilon) 
\rvert_\mathrm{HS}^p
$$
for all $p \ge 2$. Hence, taking expectations with respect to $X$ on both sides, 
we have
\begin{equation}
\label{Schr1}
\mathbb{E}\, \lvert f^\varepsilon(X, \varepsilon) \rvert^p \le 
(52\, p)^{pd/2}\, \mathbb{E}\, \lvert \mathfrak{h}^{\varepsilon(d)}
f^\varepsilon(X, \varepsilon) \rvert_\mathrm{HS}^p.
\end{equation}

It follows from a result by de la Pe\~{n}a and Gin\'e \cite{D-G}, Theorem 3.5.3 
(also see Joly and Lugosi \cite{J-L}, Theorem 8) that
\begin{equation}
\label{Schr1a}
\mathbb{E}\, \lvert f(X) \rvert^p \le 
\tilde{c}^p\, \mathbb{E}\, \lvert f^\varepsilon(X, \varepsilon) \rvert^p
\end{equation}
with some constant $\tilde{c}$ depending on $d$ only. Moreover, 
for any $i_1 \ne \ldots \ne i_d$, it is not hard to see that
$$
\mathfrak{h}^\varepsilon_{i_1 \ldots i_d} f^\varepsilon(X, \varepsilon) = 
|h(X_{i_1}, \ldots, X_{i_d})|
$$
(cf. the proof of Proposition \ref{multPol} and note that 
$\mathfrak{h}\varepsilon_i = 1$). Consequently,
\begin{equation}
\label{Schr1b}
\lvert \mathfrak{h}^{\varepsilon(d)}f^\varepsilon(X, \varepsilon) 
\rvert_\mathrm{HS} \, \le \, C_d(f),
\end{equation}
where
$$
C_d = \frac{1}{\binom{n}{d}}\Big(\sum_{i_1 \ne \ldots \ne i_d} 
\lVert h(X_1, \ldots, X_d) \rVert_\infty^2 \Big)^{1/2}.
$$
Applying \eqref{Schr1a} and \eqref{Schr1b} on \eqref{Schr1} and taking 
$p$-th roots, we arrive at
$$
\lVert f \rVert_p \le \tilde{c} \,(52\, p)^{d/2}\, C_d(f).
$$

From here on, the proof is similar to the proof of Theorem \ref{allgem} 
if we normalize $f$ such that $C_d(f) \le 1$. However, it follows from 
the assumptions on $h$ that $C_d(f) \le \hat{c} n^{-d/2}$ for some 
numerical constant $\hat{c}$ depending on $d$ and $M$ only. Hence 
we arrive at the normalization $\hat{c}^{-1} n^{d/2} f$ which yields
Corollary \ref{U-statistics}.
\end{proof}

\vskip2mm
\begin{proof}[Proof of Corollary \ref{KorrTails}]
First note that, by Chebychev's inequality, for any $p \ge 1$
\begin{equation}\label{tailsexp}
\mathbb{P}(|f| \ge e\, \lVert f \rVert_p) \le e^{-p}.
\end{equation}
Moreover, if $p \ge 2$, it follows from \eqref{byit} that
$$
e\,\lVert f \rVert_p  \, \le \, e\,\Big(\sum_{k=1}^{d-1} (41p)^{k/2} \,
\lVert \mathfrak{h}^{(k)}f \rVert_{\mathrm{HS}, 2} + 
(41p)^{d/2} \, \lVert \mathfrak{h}^{(d)}f 
\rVert_{\mathrm{HS},\infty}\Big).
$$
Here we have used that $32\,\kappa < 41$.
Assuming $\eta_f(t) \ge 2 \cdot 41$, we therefore arrive that
$$
e\,\lVert f \rVert_{\eta_f(t)/41} \le e\,\Big(\sum_{k=1}^{d-1} t + t\Big)
= (de)\, t.
$$
Hence, applying \eqref{tailsexp} to $p = \eta_f(t)/41$ (if $p \ge 2$) yields
$$
\mu(|f| \ge (de) t) \, \le \, \mu(|f| \ge e\, \lVert f \rVert_{\eta_f(t)/41})
 \, \le \, \exp\{-\eta_f(t)/41\}.
$$
Using a trivial estimate in case of $p < 2$, we also obtain that
$$
\mu(|f| \ge (de) t) \le e^2 \exp\{-\eta_f(t)/41\}.
$$
The proof is now easily completed by rescaling $f$ and
using $\eta_{(de) f}(t) \ge \eta_f(t)/(de)^2$.
\end{proof}

\vskip2mm
\begin{proof}[Proof of Theorem \ref{allgemsup}]
First note that
$$
\mathfrak{h}_i \sup_{f \in \mathfrak{F}} |f(X)| \le \sup_{f \in \mathfrak{F}}
\mathfrak{h}_i f(X)
$$
by the reverse triangular inequality, and therefore 
(writing $\mathfrak{h}^{*} (\mathfrak{F})
\equiv \mathfrak{h}^{*(1)}(\mathfrak{F})$)
$$
|\mathfrak{h} \sup_{f \in \mathfrak{F}} |f(X)| | \le |\mathfrak{h}^{*} (\mathfrak{F})|.
$$
In a similar way, we may also prove an analogue of \eqref{Rekursion}, i.\,e.
$$
|\mathfrak{h}|\mathfrak{h}^{*(d-1)}(\mathfrak{F})|_\mathrm{HS}| \le 
|\mathfrak{h}^{*(d)}(\mathfrak{F})|_\mathrm{HS}.
$$
To see this, note that $\sup_{f \in \mathfrak{F}}\,\lVert \cdot\rVert_{j,\infty}$ 
is a pseudo-norm. In view of these elementary facts, the proof
of Theorem \ref{allgemsup} is now similar to the proof of Theorem~\ref{allgem}.
\end{proof}

\vskip2mm
\begin{proof}[Proof of Theorem \ref{proppartsum}]
The proof is obtained by calculating the differences of first and second order.
To start, note that for any $\nu = 1, \ldots, n$,
$$
(\mathfrak{h}_\nu S_f(X))^2 \, = \,
\frac{1}{4}\, \Big\lVert \sum_{i \ge \nu} 
\Big(f\Big(\sum_{j=1}^{i} X_j\Big) - 
f\Big(\sum_{j=1}^{i}T_\nu X_j\Big) \Big)\Big\rVert_\infty^2
 \, \le \, \lVert f \rVert_\infty^2(n-\nu+1)^2,
$$
and consequently
$$
|\mathfrak{h} S_f(X)|^2 \le \lVert f \rVert_\infty^2 \sum_{\nu=1}^{n}
(n-\nu+1)^2 = \frac{1}{6} n(n+1)(2n+1) \lVert f \rVert_\infty^2
\le C n^3 \lVert f \rVert_\infty^2
$$
for some constant $C > 0$.
Moreover, for any $\nu \ne \mu$, $(\mathfrak{h}_{\nu \mu} S_f(X))^2$
is given by
$$
\frac{1}{16}\, 
\Big\lVert\sum_{i \ge \nu \vee \mu}\Big(f(\sum_{j=1}^{i} X_j\Big) - 
f\Big(\sum_{j=1}^{i} T_\nu X_j\Big) - 
f\Big(\sum_{j=1}^{i} T_\mu X_j\Big) +
f\Big(\sum_{j=1}^{i} T_{\nu \mu} X_j\Big)\Big)\Big\rVert^2.
$$
By similar arguments as above, this expression
does not exceed $\lVert f \rVert_\infty^2(n-(\nu \vee \mu)+1)^2$, and therefore
$$
\lVert\mathfrak{h}^{(2)} S_f(X)\rVert_\mathrm{HS}^2 = 
\sum_{\nu \ne \mu} (\mathfrak{h}_{\nu \mu} S_f(X))^2 \le 
C n^4 \lVert f \rVert_\infty^2.
$$
Combining these arguments and applying Corollary \ref{KorrTails} completes the proof.
\end{proof}

\vskip5mm
\section{Higher order Efron--Stein inequality: Proofs}

Let us first collect some elementary facts about the difference operators 
introduced in Section 1. As before, assume that $X = (X_1,\dots,X_n)$ 
has independent components.

\vskip4mm
\begin{bemerkung}
\label{DiskrDiff}
\begin{enumerate}[(i)]
Let $i = 1,\dots,n$.
\item 
If $\varepsilon = (\varepsilon_1, \ldots, \varepsilon_n)$ 
has independent Rademacher components, then
$
\mathfrak{D}_if(\varepsilon) = \frac{1}{2}\,
(f(\varepsilon)-f(\sigma_i\varepsilon)),
$
where
$\sigma_i\varepsilon = 
(\varepsilon_1, \ldots, -\varepsilon_i, \ldots, \varepsilon_n)$.
Moreover,
$$
\mathfrak{h}_i f(\varepsilon) = \mathfrak{v}_i f(\varepsilon) = 
\mathfrak{d}_i f(\varepsilon) = |\mathfrak{D}_if(\varepsilon)|.
$$
\item 
If $f(X) \in L^2(\mathbb{P})$, then
$$
(\mathfrak{v}_i f(X))^2 = \mathbb{E}_i (\mathfrak{D}_i f(X))^2\qquad \text{and}\qquad 
(\mathfrak{v}_i f(X))^2 = \mathbb{E}_i (\mathfrak{d}_i f(X))^2.
$$
In particular, we immediately obtain \eqref{L2normengleich}, where 
the identities involving $\mathfrak{d}^\pm f$ follow from symmetry and Fubini's 
theorem.

\item 
Let $f(X) \in L^2(\mathbb{P})$. Then, by independence, we can rewrite 
$\mathfrak{d}_if(X)$ as
\begin{align}
\label{d2}
\mathfrak{d}_if(X) 
  & = 
\Big(\frac{1}{2}\,\big((f(X) - \mathbb{E}_i\, f(X))^2 + 
\mathbb{E}_i\, (f(X) - \mathbb{E}_if(X))^2\big)\Big)^{1/2}\notag\\
	& = 
\Big(\frac{1}{2}\,\big((\mathfrak{D}_if(X))^2 + 
\mathbb{E}_i(\mathfrak{D}_if(X))^2\big)\Big)^{1/2}.
\end{align}

\item 
By induction over $n$, $f$ is bounded if and only if $|\mathfrak{D}f|$ 
is bounded. Using \eqref{d2}, the same holds for $|\mathfrak{d} f|$ and
$|\mathfrak{h} f|$ instead of $|\mathfrak{D}f|$.

\item 
We have $|\mathfrak{d}^+ f| \le |\mathfrak{d} f|$, 
$|\mathfrak{d}^-f| \le |\mathfrak{d} f|$, $|\mathfrak{d}^+ f| \le 
\sqrt{2}\,|\mathfrak{h}^+ f|$,
$|\mathfrak{d}^-f| \le \sqrt{2}\,|\mathfrak{h}^- f|$ and 
$|\mathfrak{d} f| \le \sqrt{2}\,|\mathfrak{h} f|$.
\end{enumerate}
\end{bemerkung}

\vskip2mm
The difference operator $\mathfrak{D}$ is closely related 
to the Hoeffding decomposition \eqref{Hoeffding}. Indeed, 
the representation \eqref{Hoeffding} 
follows by tensorizing the identity $\mathbb{E}_i + \mathfrak{D}_i = Id$,
so
\begin{equation}
\label{HT}
h_{i_1 \ldots i_k}(X_{i_1}, \ldots, X_{i_k}) = 
\Big(\prod_{j \notin \{i_1, \ldots i_k\}} 
\mathbb{E}_j \prod_{l \in \{i_1, \ldots i_k\}}
\mathfrak{D}_l\Big) f(X_1, \ldots, X_n).
\end{equation}

Many of the relations described in Remark \eqref{DiskrDiff} extend to 
higher order differences. In particular, for any $i_1 < \ldots < i_d$,
\begin{equation}
\label{relvDdd}
(\mathfrak{v}_{i_1 \ldots i_d} f(X))^2 = 
\mathbb{E}_{i_1 \ldots i_d}\, (\mathfrak{D}_{i_1 \ldots i_d} f(X))^2 = 
\mathbb{E}_{i_1 \ldots i_d}\, (\mathfrak{d}_{i_1 \ldots i_d} f(X))^2.
\end{equation}
Furthermore, similarly to \eqref{d2}, we may rewrite \eqref{dd} as
$$
\mathfrak{d}_{i_1 \ldots i_d}\,f(X) = 
\Big(\frac{1}{2^d}\,\big((\mathfrak{D}_{i_1 \ldots i_d}\,f(X))^2 + 
\sum_{k=1}^d \sum_{1 \leq s_1 < \ldots < s_k \leq d} 
\mathbb{E}_{i_{s_1} \ldots i_{s_k}}\,
(\mathfrak{D}_{i_1 \ldots i_d}f(X))^2\big)\Big)^{1/2}.
$$
\eqref{relvDdd} implies that we always have
$$
\mathbb{E}\, (\mathfrak{D}_{i_1 \ldots i_d} f)^2 
 = 
\mathbb{E}\, (\mathfrak{v}_{i_1 \ldots i_d} f)^2 
 = 
\mathbb{E}\, (\mathfrak{d}_{i_1 \ldots i_d} f)^2.
$$
Moreover, by the symmetry and Fubini's theorem,
$$
\mathbb{E}\, (\mathfrak{d}^+_{i_1 \ldots i_d} f)^2 
 = 
\mathbb{E}\, (\mathfrak{d}^-_{i_1 \ldots i_d} f)^2
 = 
\frac{1}{2}\, \mathbb{E}\, (\mathfrak{d}_{i_1 \ldots i_d} f)^2.
$$
In particular,
\begin{equation}
\label{L2normengleichd}
\mathbb{E}\, |\mathfrak{D}^{(d)} f|^2 
 = 
\mathbb{E}\, |\mathfrak{v}^{(d)} f|^2 
 = 
\mathbb{E}\, |\mathfrak{d}^{(d)} f|^2
 = 
2\, \mathbb{E}\, |\mathfrak{d}^{+(d)} f|^2
 = 
2\, \mathbb{E}\, |\mathfrak{d}^{-(d)} f|^2.
\end{equation}
Finally, as in Remark \ref{DiskrDiff} $(i)$, we may conclude that 
even the identity
$$
\mathfrak{h}_{i_1 \ldots i_d}f(\varepsilon) 
 = 
\mathfrak{v}_{i_1 \ldots i_d}f(\varepsilon) 
 = 
\mathfrak{d}_{i_1 \ldots i_d}f(\varepsilon) 
 = 
|\mathfrak{D}_{i_1 \ldots i_d}f(\varepsilon)|
$$
holds for independent Rademacher variables 
$(\varepsilon_1, \ldots, \varepsilon_n) = \varepsilon$.

The proof of Theorem \ref{HigherOrderEfronStein} is based on $L^2$-identities 
together with some kind of ``harmonic'' analysis arguments on the symmetric group. 
To this end, we shall need specific (higher order) operators $\mathfrak{L}_d$ 
we would call powers of ``Laplacians''. Here we make use of the difference 
operators $\mathfrak{D}_i$. That is, we set
\begin{equation}
\mathfrak{L}_d \ = 
\sum_{1 \le i_1 \ne i_2 \ne \ldots \ne i_d \le n} 
\mathfrak{D}_{i_1} \ldots \mathfrak{D}_{i_d}, \qquad d \in \mathbb{N}.
\label{Laplace}
\end{equation}
In case of $d = 1$ this just means summing over all $i = 1, \ldots, n$.

To motivate the notation of $\mathfrak{L}_d$, recall that for 
the discrete hypercube $\{\pm1 \}^n$, 
$\mathfrak{L}_1 = \sum_{i=1}^{n} \mathfrak{D}_i$ is the usual graph Laplacian. 
The higher order operators $\mathfrak{L}_d$ can be written as
polynomials in $\mathfrak{L}_1$ of total degree $d$. Note that 
$(\mathfrak{L}_1)^d$ can be 
expressed as a sum of $d$-th order
differences $\mathfrak{D}_{i_1 \ldots i_d}$. Hence it easily follows that
$\mathfrak{L}_d$ can be  
expressed in terms of $(\mathfrak{L}_1)^d$ by
removing all the differences in which some indexes appear more than once.
This is in accordance with our definition of the hyper-matrices
$\mathfrak{D}^{(d)}f$ (cf. the discussion of \eqref{Hessediskret}).
Relating the Hoeffding decomposition to the Laplacian $\mathfrak{L}_d$ 
yields the following result.

\vskip5mm
\begin{satz}
\label{diagonal}
Let $f = f(X_1, \ldots,\linebreak[3] X_n)$ be in $L^1(\mathbb{P})$
with Hoeffding decomposition $f = \sum_{k=0}^n f_k$. Then,
\begin{equation*}
\mathfrak{L}_d f_k = (k)_d\, f_k,
\end{equation*}
where $\mathfrak{L}_d$ is the $d$-th order Laplacian \eqref{Laplace}, 
and $(k)_d = k(k-1) \cdots (k-d+1)$. Thus, the $k$-th Hoeffding term 
is an eigenfunction of $\mathfrak{L}_d$ with eigenvalue $(k)_d$.
\end{satz}

\vskip2mm
Consequently, there is an orthogonal decomposition of $L^2$-functions $f(X)$ 
on which the Laplacian $\mathfrak{L}_d$ operates diagonally, and the eigenvalues 
of the Hoeffding terms of order up to $d-1$ are $0$.

\vskip2mm
\begin{proof}
Write 
$f_k(X_1, \ldots, X_n) = 
\sum_{j_1< \ldots < j_k} h_{j_1 \ldots j_k}(X_{j_1}, \ldots, X_{j_k})$ as 
in \eqref{Hoeffding}. Fix $j_1 < \ldots < j_k$. Then, we get
$$
\mathbb{E}_i h_{j_1 \ldots j_k}(X_{j_1} \ldots, X_{j_k}) = 
\begin{cases} 
0, & i \in \{j_1, \ldots, j_k\}, \\
h_{j_1 \ldots j_k}(X_{j_1}, \ldots, X_{j_k}), & i \notin \{j_1, \ldots, j_k\}. 
\end{cases}
$$
Therefore,
\begin{equation}
\label{D_i f_d}
\mathfrak{D}_if_d(X_1, \ldots, X_n) = 
\sum_{\substack{j_1< \ldots < j_k \\ i \in \{j_1, \ldots, j_k\}}} 
h_{j_1 \ldots j_k}(X_{j_1},\ldots, X_{j_k})
\end{equation}
and consequently, by iteration,
\begin{equation}
\label{D_ij f_d}
\mathfrak{D}_{i_1 \ldots i_d}f_k(X_1, \ldots, X_n) = 
\sum_{\substack{j_1< \ldots < j_k \\ i_1, \ldots, i_d \in \{j_1, \ldots, j_k\}}}
h_{j_1 \ldots j_d}(X_{j_1}, \ldots, X_{j_k}).
\end{equation}

It remains to check how often each term 
$h_{j_1 \ldots j_d}(X_{j_1}, \ldots, X_{j_d})$ appears in 
$\mathfrak{L}_d f_k = 
\sum_{i_1 \neq \ldots \neq i_d} \mathfrak{D}_{i_1 \ldots i_d} f_k$. 
As we just saw, each $d$-tuple $i_1 \neq \ldots \neq i_d$ such that 
$i_1, \ldots, i_d \in \{j_1, \ldots, j_k\}$ replicates the summand 
$h_{j_1 \ldots j_d}(X_{j_1}, \ldots, X_{j_d})$ precisely once. As there are 
$k (k-1) \cdots (k-d+1) = (k)_d$ such tuples, we arrive at the result.
\end{proof}

\vskip2mm
There is a ``partial integration'' formula involving difference 
operators. Recall that by \cite{G-S}, Lemma 5.1, the difference operators
$\mathfrak{D}_i$ are self-adjoint in the sense that 
$\mathbb{E}\, (\mathfrak{D}_i f)g = \mathbb{E} f(\mathfrak{D}_i g) = 
\mathbb{E}\, (\mathfrak{D}_i f)(\mathfrak{D}_i g)$ whenever 
$f(X)$ and $g(X)$ are in $L^2(\mathbb{P})$. In particular, 
for the Laplacians $\mathfrak{L}_d$ from \eqref{Laplace}, we have
\begin{equation}
\label{self-adjoint}
\mathbb{E}\, (\mathfrak{L}_d f) g \, = \, \mathbb{E} f (\mathfrak{L}_d g) \, = 
\sum_{i_1 \ne \ldots \ne i_d} \mathbb{E}\, (\mathfrak{D}_{i_1 \ldots i_d}f)
(\mathfrak{D}_{i_1 \ldots i_d}g).
\end{equation}
This identity can be used to obtain the following relation:

\vskip5mm
\begin{proposition}
\label{Gradient-Hesse}
Let $f(X) \in L^2(\mathbb{P})$ have the Hoeffding decomposition
$f = \sum_{m=d}^{n} f_m$. If $k \le d \le n$, then
$$
\mathbb{E}\, |\mathfrak{D}^{(k-1)} f|^2 \, \le \,
\frac{1}{d-k+1}\, \mathbb{E}\, \lvert \mathfrak{D}^{(k)} f \rvert^2.
$$
Equality holds only if $f = f_d$.
\end{proposition}

\vskip2mm
\begin{proof}
First, let $f = f_m$. Then, applying \eqref{self-adjoint} leads to
$$
\mathbb{E}\, \lvert \mathfrak{D}^{(k)} f_m \rvert^2 \, = 
\sum_{i_1 \ne \ldots \ne i_k} \mathbb{E}\, (\mathfrak{D}_{i_1 \ldots i_k}f_m)
(\mathfrak{D}_{i_1 \ldots i_k} f_m) = \mathbb{E} f_m (\mathfrak{L}_k f_m).
$$
Moreover, Theorem \ref{diagonal} yields $\mathfrak{L}_k f_m = (m)_k f_m$. 
Consequently,
$$
\mathbb{E}\, \lvert \mathfrak{D}^{(k)} f_m \rvert^2 = 
(m)_k\, \mathbb{E} f_m^2. \qquad (*)
$$
The same argument with $k$ replaced by $k-1$ yields
$$
\mathbb{E}\, \lvert \mathfrak{D}^{(k-1)} f_m \rvert^2 = 
(m)_{k-1} \mathbb{E} f_m^2. \qquad (**)
$$
Comparing $(*)$ and $(**)$ completes the proof in the case $f = f_m$.

For functions with arbitrary Hoeffding decomposition we shall use 
the orthogonality of the terms in the Hoeffding decomposition and obtain
$$
\mathbb{E}\, |\mathfrak{D}^{(k-1)} f|^2 = 
\sum_{m=d}^n \frac{1}{m-k+1}\, \mathbb{E}\, \lvert \mathfrak{D}^{(k)} f_m \rvert^2
\le \frac{1}{d-k+1}\, \mathbb{E}\, \lvert \mathfrak{D}^{(k)} f \rvert^2.
$$
\end{proof}

\vskip2mm
\begin{proof}[Proof of Theorem \ref{HigherOrderEfronStein}]
Due to \eqref{L2normengleichd}, it suffices to prove 
Theorem \ref{HigherOrderEfronStein} for the difference operator $\mathfrak{D}$. 
In this case, iterating Proposition \ref{Gradient-Hesse} yields 
the result.
\end{proof}

\vskip2mm
Finally, we prove \eqref{Houdre-Formel}. By orthogonality and
Proposition \ref{Gradient-Hesse},
$$
\mathbb{E}\, \lvert \mathfrak{v}^{(k)} f \rvert^2 = 
\sum_{j=k}^{n} \mathbb{E}\, \lvert \mathfrak{v}^{(k)} f_j \rvert^2 = 
\sum_{j=k}^{n} \frac{1}{(j-k)!}\, \mathbb{E}\, \lvert \mathfrak{v}^{(j)}
f_j \rvert^2
$$
for any $k = 1, \ldots, n$, which may be rewritten as
$$
\mathbb{E}\, \lvert \mathfrak{v}^{(k)} f_k \rvert^2 = \mathbb{E}\,
\lvert \mathfrak{v}^{(k)} f \rvert^2 - \sum_{j=k+1}^{n} \frac{1}{(j-k)!}\,
\mathbb{E}\, \lvert \mathfrak{v}^{(j)} f_j \rvert^2.
$$
Iteratively plugging this into \eqref{Varianzdarst}, we obtain that
$$
\mathrm{Var} f = \sum_{k=1}^{n} R_k \mathbb{E}\, \lvert
\mathfrak{v}^{(k)} f \rvert^2,\qquad R_k = 
\sum_{j=0}^{k-1} R_j \frac{1}{(k-j)!},\qquad R_0 := 1.
$$
It follows that for $k \ge 1$, $R_k = (-1)^{k+1}/k!$ which finishes the
proof.

\vskip5mm
\section{Differentiable Functions: Proofs}

Given a continuous function on an open subset $G \subset \mathbb{R}^n$, 
the equality
\begin{equation}
\label{generalizedmodulus}
|\nabla f(x)| = \limsup_{x \to y} \frac{|f(x) - f(y)|}{|x-y|}, \qquad x \in G,
\end{equation}
may be used as definition of the generalized modulus of the gradient
of $f$. The function $|\nabla f|$ is Borel measurable, and if $f$ is 
differentiable at $x$, $|\nabla f(x)|$ agrees with the Euclidean norm of the 
usual gradient. This operator preserves many identities from calculus in form 
of inequalities, such as a ``chain rule inequality''
\begin{equation}
\label{chainrule}
|\nabla T(f)| \le |T'(f)||\nabla f|,
\end{equation}
where $|T'|$ is understood according to \eqref{generalizedmodulus} again.

Using the generalized modulus of the gradient, there is an analogue of 
Lemma \ref{itGradHessindep} for the operator norms of the derivatives of 
consecutive orders:

\vskip5mm
\begin{lemma}
\label{itGradHess}
Given a $\mathcal{C}^d$-smooth function 
$f \colon G \to \mathbb{R}$, $d \in \mathbb{N}$, at all points $x \in G$,
$$
|\nabla \lvert f^{(d-1)}(x) \rvert_\mathrm{Op}| \le 
\lvert f^{(d)}(x) \rvert_\mathrm{Op}.
$$
\end{lemma}

\vskip2mm
\begin{proof}
Indeed, for any $h \in \mathbb{R}^n$, by the triangle inequality,
\begin{align*}
 & 
\big|\,\lvert f^{(d-1)}(x+h) \rvert_\mathrm{Op} - 
\lvert f^{(d-1)}(x) \rvert_\mathrm{Op}\big|
 \, \le \,
\lvert f^{(d-1)}(x+h) - f^{(d-1)}(x) \rvert_\mathrm{Op}\\
 = \, &
\sup \{ (f^{(d-1)}(x+h) - f^{(d-1)}(x))[v_1, \ldots, v_{d-1}] 
\colon v_1, \ldots, v_{d-1} \in S^{n-1} \},
\end{align*}
while, by the Taylor expansion,
$$
(f^{(d-1)}(x+h) - f^{(d-1)}(x))[v_1, \ldots, v_{d-1}] = 
f^{(d)}(x)[v_1, \ldots, v_{d-1}, h] + o(|h|)
$$
as $h \to 0$. Here, the $o$-term can be bounded by a quantity which is 
independent of $v_1, \ldots, v_{d-1} \in S^{n-1}$. As a consequence,
\begin{align*}
 &
\limsup_{h \to 0} 
\frac{|\,\lvert f^{(d-1)}(x+h) \rvert_\mathrm{Op} - \lvert f^{(d-1)}(x) 
\rvert_\mathrm{Op}|}{|h|}\\
 \le \; &
\sup \{ f^{(d)}(x)[v_1, \ldots, v_{d-1}, v_d] \colon v_1, \ldots, v_d \in S^{n-1}\}
 = \lvert f^{(d)}(x) \rvert_\mathrm{Op}.
\end{align*}
\end{proof}

\vskip2mm
For higher order concentration we need to establish a recursion for 
the $L^p$-norms of the derivatives 
of $f$ of consecutive orders similarly to \eqref{Ineq1}. To this end, 
we recall a classical result on the moments of Lipschitz functions in the
presence of a logarithmic Sobolev inequality which goes back to Aida and 
Stroock \cite{A-S}. Namely, if a probability measure $\mu$ on $G$ 
satisfies a logarithmic Sobolev inequality \eqref{LSI1} with constant 
$\sigma^2$, then, for any locally Lipschitz function
$g \colon G \to \mathbb{R}$, and any $p > 2$,
\begin{equation}\label{moments}
\lVert g \rVert_p^2  \, \le \, \lVert g \rVert_2^2 +
\sigma^2(p-2)\, \lVert \nabla g \rVert_p^2.
\end{equation}
For the reader's convenience, let us briefly recall the argument. 
We may assume $g$ to be bounded, in which case the squares of
the $L_p(\mu)$-norms of $g$ have derivatives
\begin{equation}\label{Abl}
\frac{d}{dp}\, \lVert g \rVert_p^2 = 
\frac{2}{p^2}\, \lVert g \rVert_p^{2-p}\, \text{Ent}_\mu(|g|^p).
\end{equation}
We apply this identity to the function $u = |g|^{p/2}$. By the chain rule inequality 
\eqref{chainrule}, $|\nabla u|^2 \le \frac{p^2}{4}\, |g|^{p-2}\, |\nabla g|^2$.
Hence, by H\"older's inequality,
$$
\int |\nabla u|^2\, d\mu \le 
\frac{p^2}{4} \Big(\int|g|^p\,d\mu\Big)^{\frac{p-2}{p}}
\Big(\int |\nabla g|^p\,d\mu\Big)^{\frac{2}{p}} = 
\frac{p^2}{4} \lVert g \rVert_p^{p-2}\lVert \nabla g \rVert_p^2.
$$
Applying \eqref{LSI1} to the function $u$, we therefore obtain
$$
\text{Ent}_\mu (|g|^p) = \text{Ent}_\mu (u^2) \le 
2 \sigma^2 \int |\nabla u|^2\, d\mu \le 
\frac{p^2 \sigma^2}{2}\, \lVert g \rVert_p^{p-2}\, \lVert \nabla g \rVert_p^2.
$$
Combining this with \eqref{Abl}, we arrive at the differential inequality 
$\frac{d}{dp}\, \lVert g \rVert_p^2 \le \sigma^2 \lVert \nabla g \rVert_p^2$. 
Integrating it from $2$ to $p$ yields \eqref{moments}.

Combining Lemma \ref{itGradHess} and \eqref{moments}, we are now able 
to prove Theorem \ref{kontinuierlich}.

\vskip2mm
\begin{proof}[Proof of Theorem \ref{kontinuierlich}]
Using \eqref{moments} with $f$ replaced by 
$\lvert f^{(k-1)} \rvert_\mathrm{Op}$, $1 \le k \leq d$, we get
\begin{align}
\label{MomenteiteriertdiffA}
\begin{split}
\lVert f^{(k-1)} \rVert_{\mathrm{Op}, p}^2 
 & \le 
\lVert f^{(k-1)} \rVert_{\mathrm{Op}, 2}^2 + 
\sigma^2 (p-2)\, \lVert \nabla \lvert f^{(k-1)} \rvert_\mathrm{Op} \rVert_p^2\\
 & \le 
\lVert f^{(k-1)} \rVert_{\mathrm{Op}, 2}^2 + 
\sigma^2 (p-2)\, \lVert f^{(k)} \rVert_{\mathrm{Op}, p}^2,
\end{split}
\end{align}
where  Lemma \ref{itGradHess} was applied on the last step. Consequently, 
by iteration,
\begin{align}
\lVert f \rVert_p^2 
 & \ \le \
\lVert f \rVert_2^2 + \sum_{k=1}^{d-1} (\sigma^2 (p-2))^k\, \lVert f^{(k)} 
\rVert_{\mathrm{Op}, 2}^2	+ 
(\sigma^2 (p-2))^d\, \lVert f^{(d)} \rVert_{\mathrm{Op}, p}^2\notag\\
 & \ \le \
\sigma^2\, \lVert \nabla f \rVert_2^2 + 
\sum_{k=1}^{d-1} (\sigma^2 (p-2))^k\, \lVert f^{(k)} \rVert_{\mathrm{Op}, 2}^2 + 
(\sigma^2 (p-2))^d\, \lVert f^{(d)} \rVert_{\mathrm{Op}, p}\notag\\
 & \ \le \ \sum_{k=1}^{d-1} (\sigma^2 p)^k\, \lVert f^{(k)} 
\rVert_{\mathrm{Op}, 2}^2 + 
(\sigma^2 p)^d\, \lVert f^{(d)} \rVert_{\mathrm{Op}, p}^2.\label{pfstep}
\end{align}
Here, the second step was based on the Poincar\'e-type inequality.
Since 
$\lVert f^{(k)} \rVert_{\mathrm{Op}, 2}^2 \le \min(1, \sigma^{2(d-k)})$ 
for all $k = 1, \ldots, d-1$ and 
$\lVert f^{(d)} \rVert_{\mathrm{Op}, \infty} \le 1$
by assumption, we obtain
\begin{equation}
\label{MomenteiteriertdiffB}
\lVert f \rVert_p^2 \, \le \, \sigma^{2d} \sum_{k=1}^{d} p^k \, \le \,
\frac{1}{1 - p^{-1}}\, (\sigma^2 p)^d \, \le \, 2\, (\sigma^2 p)^d
\end{equation}
and therefore
$
\lVert f \rVert_p \le (2 \sigma^2 p)^{d/2}
$
for all $p \ge 2$. Moreover,
$
\lVert f \rVert_p \le \lVert f \rVert_2 \le (4 \sigma^2)^{d/2}
$
for $p < 2$. It follows that
$
\lVert \lvert f \rvert^{2/d} \rVert_k \le \gamma k
$
for all $k \in \mathbb{N}$, i.\,e. \eqref{subexp} with
$\gamma = 4 \sigma^2$.
\end{proof}

\vskip2mm
\begin{proof}[Proof of Theorem \ref{kontinuierlichmAbl}]
Starting as in the proof of Theorem \ref{kontinuierlich}, we arrive at
\begin{align}
\lVert f \rVert_p^2 
 & \le 
\sum_{k=1}^{d-1} (\sigma^2 p)^k\, \lVert f^{(k)} \rVert_{\mathrm{HS}, 2}^2 + 
(\sigma^2 p)^d\, \lVert f^{(d)} \rVert_{\mathrm{Op}, p}^2,\label{MomenteiteriertA}
\end{align}
where we used the property that the operator norms are dominated by 
the Hilbert--Schmidt norms. Moreover, since 
$\int_G \partial_{i_1 \ldots i_k} f\, d\mu = 0$,
by the Poincar\'e-type inequality,
$$
\int_G (\partial_{i_1 \ldots i_k} f)^2\, d\mu \le 
\sigma^2 \sum_{j=1}^{n} \int_G (\partial_{i_1 \ldots i_k j} f)^2\, d\mu
$$
whenever $1 \le i_1, \ldots, i_k \le n$, $k \leq d-1$. Summing over all 
$1 \le i_1, \ldots, i_k \le n$, we get
\begin{equation}
\label{MomenteiteriertB}
\lVert f^{(k)} \rVert_{\mathrm{HS}, 2}^2 = 
\int_G \lvert f^{(k)} \rvert_\mathrm{HS}^2\, d\mu \le 
\sigma^2 \int_G \lvert f^{(k+1)} \rvert_\mathrm{HS}^2\, d\mu = 
\sigma^2\, \lVert f^{(k+1)} \rVert_{\mathrm{HS}, 2}^2.
\end{equation}
Using \eqref{MomenteiteriertB} in \eqref{MomenteiteriertA} and iterating, 
we thus obtain
\begin{equation*}
\lVert f \rVert_p^2 \le 
\sum_{k=1}^{d-1} \sigma^{2d} p^k\, \lVert f^{(d)} 
\rVert_{\mathrm{HS}, 2}^2 + 
(\sigma^2 p)^d\, \lVert f^{(d)} \rVert_{\mathrm{Op}, p}^2.
\end{equation*}
Noting that $\lVert f^{(d)} \rVert_{\mathrm{HS}, 2} \le 1$ and 
$\lVert f^{(d)} \rVert_{\mathrm{Op}, \infty} \le 1$, we arrive at 
\eqref{MomenteiteriertdiffB},
from where we may proceed as in the proof of Theorem \ref{kontinuierlich}.
\end{proof}

\vskip2mm
\begin{proof}[Proof of Corollary \ref{KorrTailskont}]
For any $p \ge 2$, it follows from \eqref{pfstep} that
$$
e\,\lVert f \rVert_p  \, \le \, e\,\Big(\sum_{k=1}^{d-1} (\sigma^2 p)^{k/2}\,
\lVert f^{(k)} \rVert_{\mathrm{Op}, 2} + 
(\sigma^2 p)^{d/2}\, \lVert f^{(d)} \rVert_{\mathrm{Op}, \infty}\Big).
$$
From here we may proceed as in the proof of Corollary \ref{KorrTails}.
\end{proof}

\vskip5mm
\section{Functions on the Sphere: Proofs}

First, let us recall some basic facts about the spherical calculus
(cf. \cite{S-W} or e.\,g. \cite{B-C-G}). 
The normalized Lebesgue measure $\sigma_{n-1}$ on the unit sphere
$S^{n-1}$ can be introduced as the distribution of $Z/|Z|$, assuming 
that the random vector $Z$ has a standard normal distribution in 
$\mathbb{R}^n$. Using independence of $|Z|$ and $Z/|Z|$, 
this description implies that for any $p$-homogeneous 
function $f \colon \mathbb{R}^n \setminus \{0\} \to \mathbb{R}$,
\begin{equation}
\label{p-homog}
\int_{\mathbb{R}^n} f(x)\, d\gamma_n(x) = 
\mathbb{E}\,|Z|^p \int_{S^{n-1}} f(\theta)\, d\sigma_{n-1}(\theta)
\end{equation}
provided that all the integrals involved exist.

A function $f$ on $S^{n-1}$ is called $\mathcal{C}^d$-smooth if it can be 
extended to a $\mathcal{C}^d$-smooth function on some open subset $G$ of 
$\mathbb{R}^n$ containing the unit sphere. If $f$ is $\mathcal{C}^1$-smooth 
on $S^{n-1}$, then at every point $\theta \in S^{n-1}$ it admits 
the Taylor expansion
\begin{equation}\label{Taylor}
f(\theta') = f(\theta) + \langle v, \theta' - \theta \rangle + 
o(|\theta' - \theta|)\qquad \text{as $\theta' \to \theta$,}\quad \theta' \in S^{n-1}
\end{equation}
with some $v \in \mathbb{R}^n$. Among all the vectors $v$ fulfilling 
\eqref{Taylor}, the one with smallest Euclidean norm represents
the spherical derivative or gradient of $f$ at $\theta$ and is denoted 
$\nabla_S f (\theta)$. Equivalently, in terms of the usual (Euclidean) 
gradient, we have
$$
\nabla_S f (\theta) \, = \, P_{\theta^\perp} \nabla f (\theta) \, = \,
\nabla f (\theta) - \langle \nabla f (\theta), \theta \rangle\, \theta,
$$
where $P_{\theta^\perp}$ denotes the orthogonal projection from $\mathbb{R}^n$ 
to the tangent space $\theta^\perp$. In particular,
$|\nabla_S f (\theta)| \le |\nabla f(\theta)|$ for all $\theta \in S^{n-1}$. 
If $f$ is a $\mathcal{C}^1$-function on $S^{n-1}$, the norm
of the spherical gradient $|\nabla_S f|$ coincides with the generalized 
modulus of the gradient  \eqref{generalizedmodulus} using
either the geodesic or the induced Euclidean distance on $S^{n-1}$.

By a result of Mueller and Weissler \cite{M-W}, the uniform measure 
$\sigma_{n-1}$ on $S^{n-1}$ satisfies a logarithmic Sobolev inequality with
constant $\sigma^2 = \frac{1}{n-1}$. In other words, 
\begin{equation}
\label{Mueller-Weissler}
\text{Ent}_{\sigma_{n-1}} (f^2) \le 
\frac{2}{n-1} \int |\nabla_S f|^2\, d\sigma_{n-1}
\end{equation}
for any smooth $f \colon S^{n-1} \to \mathbb{R}$.
Therefore, considering any open neighbourhood $G$ of $S^{n-1}$, 
we may regard $\sigma_{n-1}$ as a Borel probability measure on $G$ satisfying 
a logarithmic Sobolev inequality with constant $\sigma^2 = \frac{1}{n-1}$. 
Hence, Theorem \ref{sphere} directly follows from Theorem \ref{kontinuierlich}. 
It remains to note that in the notation of Theorem \ref{kontinuierlich},
$$
\min(1, \sigma^{d-k}) \ge n^{-(k-d)/2},
$$
hence arriving at the conditions used in Theorem \ref{sphere}.

In a similar way, we now prove Theorem \ref{spheremAbl}.

\vskip2mm
\begin{proof}[Proof of Theorem \ref{spheremAbl}]
It follows from 
Theorem \ref{kontinuierlichmAbl}, once the partial derivatives up to order 
$d-1$ are centered under $\sigma_{n-1}$. Indeed, we may assume that 
$f$ is defined on $\mathbb{R}^n \setminus \{0\}$ (cf. \eqref{p-homext}). 
Fix any $i_1 \le \ldots \le i_k$, $k \leq d-1$. Noting that 
$x_{i_1} \cdots x_{i_k} f(x)$ is $(p+k)$-homogeneous, by \eqref{p-homog} 
and a $k$-fold partial integration, we obtain
\begin{align*}
\int_{\mathbb{R}^n} \partial_{i_1 \ldots i_k} f(x)\, d\gamma_n(x) 
 & = 
\int_{\mathbb{R}^n} x_{i_1} \cdots x_{i_k} f(x)\, d\gamma_n(x)\\
 & = 
\mathbb{E}\,|Z|^{p+k} \int_{S^{n-1}} \theta_{i_1} \cdots \theta_{i_k} 
f(\theta)\, d\sigma_{n-1}(\theta).
\end{align*}
On the other hand, since $f$ is $p$-homogeneous, $\partial_{i_1 \ldots i_k} f$ 
is $(p-k)$-homogeneous. Therefore, applying \eqref{p-homog} again,
\begin{equation*}
\int_{\mathbb{R}^n} \partial_{i_1 \ldots i_k} f(x)\, d\gamma_n(x) = 
\mathbb{E}\,|Z|^{p-k} \int_{S^{n-1}} \partial_{i_1 \ldots i_k} 
f(\theta)\, d\sigma_{n-1}(\theta).
\end{equation*}
Hence, orthogonality to all polynomials of total degree at most $d-1$ 
implies that the partial derivatives up to order $d-1$ are centered
(if the involved integrals exist). Since $f$ is a $\mathcal{C}^d$-function 
(hence bounded on $S^{n-1}$), this holds true if $\mathbb{E}\,|Z|^{p-k} < \infty$ 
for all $k = 0, 1, \ldots, d-1$, which in turn is satisfied iff $p-(d-1)+n>0$. 
\end{proof}

\vskip2mm
To prove Theorem \ref{sphere-intr}, we need some 
further details about spherical derivatives. First note that $\nabla_S f$ 
is a vector-valued function on $S^{n-1}$, and hence we may define spherical 
partial derivatives of first and higher orders as suggested in Section 1.4.

Any function $f$ on $S^{n-1}$ can be extended to a $p$-homogeneous function 
$F$ on $\mathbb{R}^n$ (where $p \in \mathbb{R}$) by putting
\begin{equation}\label{p-homext}
F(x) = 
\begin{cases}
r^p f(\theta), & x \ne 0, \\ 0, & x = 0,
\end{cases}\qquad r = |x|,\quad \theta = x/|x|.
\end{equation}
If $f$ is $\mathcal{C}^d$-smooth, its $p$-homogeneous extension $F$ 
will be $\mathcal{C}^d$-smooth on $\mathbb{R}^n \setminus \{0\}$.

The spherical derivative $\nabla_S f$ of a $\mathcal{C}^1$-smooth function 
on $S^{n-1}$ and the derivatives of its $p$-homogeneous extensions $F$ are
related by the identity
$$
\nabla F(x) = r^{p-1}\, [p f(\theta) \theta + \nabla_S f (\theta)], \qquad
x \ne 0
$$
(see \cite{B-C-G}, Proposition 12.1). In particular, for the $0$-homogeneous 
extension $F^{(0)}(x) = f(\theta)$, $\nabla F^{(0)}$ is $(-1)$-homogeneous 
and is given by
\begin{equation}\label{Abl0homog}
\nabla F^{(0)}(x) = r^{-1} \nabla_S f(\theta),
\end{equation}
so that $\nabla F^{(0)} = \nabla_S f$ on $S^{n-1}$. In other words, 
$\partial_i F^{(0)} = D_i f$ on $S^{n-1}$, where $\partial_i$ and $D_i$ 
denote the partial and spherical partial derivatives, respectively.

By iteration, we can retrieve spherical partial derivatives of any order 
from suitable homogeneous extensions. To start, note that 
$F^{(1)}(x) = r \nabla F^{(0)}(x)$ is a $0$-homogeneous vector-valued 
function on $\mathbb{R}^n \setminus \{0\}$. It follows that for any 
$1 \le i, j \le n$, $\partial_i F^{(1)}_j$ is a $(-1)$-homogeneous function with
$$
\partial_i F^{(1)}_j(x) = \partial_i (r \partial_j F^{(0)}(x)) = 
r^{-1} D_i \langle \nabla_S f(\theta), e_j \rangle = r^{-1} D_{ij} f(\theta),
$$
so that $\partial_i F^{(1)}_j = D_{ij} f$ on $S^{n-1}$.
For general $k \in \mathbb{N}$, we may therefore define
\begin{equation}
\label{sphparder1}
F^{(k)}(x) = r \nabla F^{(k-1)}(x),
\end{equation}
which is a $0$-homogeneous function on $\mathbb{R}^n \setminus \{0\}$ 
taking values in $\mathbb{R}^{n^k}$. Arguing as above, for any 
$1 \le i_1, \ldots, i_k \le n$,
\begin{equation}
\label{sphparder2}
F^{(k)}_{i_1 \ldots i_k} = D_{i_1 \ldots i_k} f
\end{equation}
on $S^{n-1}$, where $D_{i_1 \ldots i_k} f$ denotes the $k$-th order spherical 
partial derivatives of $f$.

With the help of the $0$-homogeneous functions $F^{(k)}$, $k = 0, 1, \ldots, d$, 
it is now possible to adapt the key steps of the proof of Theorem
\ref{kontinuierlich}.

\vskip5mm
\begin{lemma}
\label{itGradHesssph}
If $f \colon S^{n-1} \to \mathbb{R}$  is $\mathcal{C}^d$-smooth, then
for all $\theta \in S^{n-1}$,
$$
|\nabla\lvert D^{(d-1)}f(\theta) \rvert\,| \le \lvert D^{(d)}f(\theta) \rvert.
$$
\end{lemma}

\vskip2mm
\begin{proof}
Noting that $|F^{(d-1)}(x)|$ is the $0$-homogeneous extension of 
$|D^{(d-1)} f(\theta)|$, we may perform similar arguments as in the proof of 
Lemma \ref{itGradHess}.
\end{proof}

\begin{proof}[Proof of Theorem \ref{sphere-intr}]
Since the uniform distribution on the sphere satisfies 
a logarithmic Sobolev inequality \eqref{Mueller-Weissler}, the moment recursion 
\eqref{moments} from the proof of Theorem \ref{kontinuierlich} remains valid 
for functions on the sphere with intrinsic derivatives.
Hence, using \eqref{moments} with $f$ replaced by 
$\lvert D^{(k-1)} f \rvert_\mathrm{Op}$ for $k = 1, \ldots, d$, we obtain
\begin{align*}
\lVert D^{(d-1)}f \rVert_{\mathrm{Op}, p}^2 
 & \le 
\lVert D^{(d-1)}f \rVert_{\mathrm{Op}, 2}^2 + 
\sigma^2 (p-2)\, \lVert \nabla^* \lvert D^{(d-1)}f \rvert_\mathrm{Op} \rVert_p^2\\
 & \le 
\lVert D^{(d-1)}f \rVert_{\mathrm{Op}, 2}^2 + 
\sigma^2 (p-2)\, \lVert D^{(d)}f \rVert_{\mathrm{Op}, p}^2.
\end{align*}
Here the last step relies upon Lemma \ref{itGradHesssph}. But this is 
the same inequality as \eqref{MomenteiteriertdiffA}. Therefore, 
the rest of the proof is analogous to the proof of 
Theorem \ref{kontinuierlich}.
\end{proof}

\vskip5mm
\section{Polynomials and Edgeworth-type Expansions on the Sphere}

For vectors $a = (a_1, \ldots, a_n) \in \mathbb{R}^n$ and an integer
$d \ge 3$, consider the polynomials
$$
Q_{d, a} (\theta) = \sum_{i=1}^{n} a_i \theta_i^d,\qquad \theta \in S^{n-1},
$$
extending them to $\mathbb{R}^n$ by setting 
$Q_{d,a}(x) = \sum_{i=1}^{n} a_i x_i^d$. By an easy calculation,
$$
\int_{S^{n-1}} \theta_i^{2p}\, d\sigma_{n-1}(\theta) = 
\frac{(2p-1)!!}{n(n+2) \cdots (n+2p-2)}, \qquad i = 1, \ldots, n, \ \ p \in \mathbb{N}.
$$
Therefore, differentiating $Q_{d,a} (x)$, it follows that,
for any $k = 1, \ldots, d-1$, 
$$
\lVert Q_{d,a}^{(k)} \rVert_{\mathrm{HS},2}^2 \, = \,
\frac{\big(2\,(d-k)-1\big)!!\,(d)_k^2}{n(n+2) \cdots (n+2\,(d-k-1))}\, 
\sum_{i=1}^{n} a_i^2,
$$
where we used the notation $(d)_k = d(d-1) \cdots (d-k+1)$. Similarly,
$$
\lvert Q_{d,a}^{(d)}(\theta) \rvert_\mathrm{HS}^2 \, = \, (d!)^2 \sum_{i=1}^{n} a_i^2.
$$
As a consequence, using the normalization $n^{-1} \sum_{i=1}^{n} a_i^2 = 1$ 
and choosing a suitable constant $c_d$, we see that the function 
$n^{-1/2} c_d\, Q_{d,a}$ satisfies the conditions of Theorem
\ref{sphere}. Summarizing, we arrive at the following result.

\vskip5mm
\begin{proposition}
\label{elementarePolynome}
Let $d \ge 3$ and let $n^{-1} \sum_{i=1}^{n} a_i^2 = 1$.
There exists some constant $c_d>0$ depending on $d$ only such that
$$
\int_{S^{n-1}} \exp \Big(c_d\, n^{(d-1)/d}\, 
|Q_{d,a} - \bar Q_{d,a}|^{2/d}\Big)\, d\sigma_{n-1} \, \le \, 2,
$$
where $\bar Q_{d,a}$ denotes the $\sigma_{n-1}$-mean of $Q_{d,a}$.
In particular,
$
Q_{d, a} - \bar Q_{d,a} = \mathcal{O}_{\sigma_{n-1}} \big(n^{-\frac{d-1}{2}}\big).
$
\end{proposition}

\vskip2mm
Note though that if $d$ is odd, $\bar{Q}_{d,a} = 0$, while for $d$ even,
\begin{equation}
\label{MomenteQ}
\bar Q_{d,a} = \frac{(d-1)!!}{(n+2) \cdots (n+d-2)}\, \bar{a}, \qquad
\bar{a} = n^{-1} \sum_{i=1}^{n} a_i.
\end{equation}

In order to illustrate possible applications of these
results, consider smooth function of the form
$$
h_n(\theta) = \mathbb{E}\, H\Big(\sum_{i=1}^n \theta_i X_i\Big), \qquad 
\theta \in S^{n-1},
$$
with $X_i \in \mathbb{R}^k$ independent and 
such that $\mathrm{Cov}(X_i) = \mathrm{Id}$. For simplicity, let us assume that 
$X_i$ is symmetric, i.\,e. $X_i = - X_i$ in distribution. It is known, 
cf. \cite{G-H}, that $h_n(\theta)$ may be approximated via Edgeworth expansions,
in particular -- by the polynomial 
$
\Gamma_0 + \sum_{i=1}^{n} \Gamma_{4,i}\, \theta_i^4,
$
as long as $H \in \mathcal{C}^4(\mathbb{R}^k)$ and assuming that the quantity
$$
M = \sup_{x \in \mathbb{R}^k} \Big(\lvert H(x) \rvert + \sup_{|\alpha| = 4}\, 
\lvert \partial^\alpha H(x) \rvert \big)\big(1 + \lvert x \rvert^6\Big)^{-1}
$$
is finite. The $\Gamma$-terms are the non-vanishing even Edgeworth expansion 
terms defined by $\Gamma_0 = \mathbb{E}H(N)$, where $N$ is a standard normal
random vector in $\mathbb{R}^k$, and by
\begin{gather*}
\Gamma_{4,i} \, = \,
\frac{1}{24} 
\left(\frac{\partial^4}{\partial \varepsilon^4}\Big|_{\varepsilon = 0}\,
\mathbb{E} H(N + \varepsilon X_j) - 3 \,
\frac{\partial^4}{\partial \varepsilon_1^2 \partial \varepsilon_2^2}
\Big|_{\varepsilon_1 = \varepsilon_2 = 0} 
\mathbb{E} H(N + \varepsilon_1 X_i + \varepsilon_2 \bar{X}_i)\right),
\end{gather*}
where $X_i$, $\bar{X}_i$, $N$ are independent and $\bar{X}_i$ denotes 
an independent copy of $X_i$.

If furthermore $\rho_{p,i}= \mathbb{E}\, |X_i|^p < \infty$ for $p \le 6$, 
then an inspection of the proof of \cite{G-H}, Theorem 3.6, yields 
an explicit bound for the error
\begin{equation}
\label{Edgeworth}
R(\theta) = 
h_n(\theta) - \Big(\Gamma_0 + \sum_{i=1}^{n} \Gamma_{4,i}\, \theta_i^4\Big),
\end{equation}
namely
\begin{equation}
\label{Absch}
|R(\theta)| \, \le \,
c_M \Big(\sum_{i=1}^{n} \rho_{6,i} \theta_i^6 + 
\Big(\sum_{i=1}^{n} \rho_{3,i} |\theta_i|^3\Big)^4\Big) \, \le \,
c_M \sum_{i=1}^{n}\, \big(\rho_{6,i} + \rho_{3,i}^4\big)\, \theta_i^6
\end{equation}
with some constant $c_M$ depending on $M$ only. 

To study the asymptotic behaviour of $h_n(\theta)$ as a function of 
$\theta \in S^{n-1}$ as $n \to \infty$, we apply Proposition \ref{elementarePolynome}
together with \eqref{MomenteQ}. Here we are interested in concentration 
inequalities for $R(\theta)$, i.\,e. we do not only center around the constant 
term $\Gamma_0$ but also include the fourth order term 
$\sum_{i=1}^{n} \Gamma_{4,i} \theta_i^4$. Indeed, write
$$
Q_{6, \rho} (\theta) = 
\sum_{i=1}^{n} \big(\rho_{6,i} + \rho_{3,i}^4\big)\, \theta_i^6,
$$
so that $|R(\theta)| \le c_M Q_{6, \rho} (\theta)$ by \eqref{Absch}. 
Dividing $Q_{6, \rho}$ by 
$\rho_* = (\frac{1}{n} \sum_{i=1}^{n} (\rho_{6,i} + \rho_{3,i}^4)^2)^{1/2}$, 
we may apply Proposition \ref{elementarePolynome} with $d = 6$, which yields 
\begin{equation}
\label{mitMittel}
\int_{S^{n-1}} \exp \Big(\frac{c_6}{\rho_*^{1/3}}\, n^{5/6}\ 
|Q_{6,\rho}(\theta) - \bar Q_{6,\rho}|^{1/3}\Big)\, d\sigma_{n-1}(\theta) 
\le 2
\end{equation}
for some absolute constant $c_6 > 0$. Furthermore, by \eqref{MomenteQ},
$$
\bar Q_{6, \rho} = \frac{15}{(n+2)(n+4)}\, \bar{\rho}, \qquad
\bar{\rho} = \frac{1}{n}\, \sum_{i=1}^{n} (\rho_{6,i} + \rho_{3,i}^4).
$$
In particular,
\begin{equation}
\label{nurMittel}
\int_{S^{n-1}} \exp \Big(\frac{c_0}{\rho_*^{1/3}}\, n^{2/3} \,
|\bar Q_{6,\rho}|^{1/3} \Big)\, d\sigma_{n-1}(\theta) \le 2
\end{equation}
for some absolute constant $c_0 > 0$. Applying the Cauchy--Schwarz inequality 
together with \eqref{mitMittel} and \eqref{nurMittel}, we therefore arrive 
at the concentration inequality
\begin{equation*}
\int_{S^{n-1}} \exp \Big(\frac{c}{\rho_*^{1/3}}\, n^{2/3} \,
|R(\theta)|^{1/3} \Big)\, d\sigma_{n-1}(\theta) \le 2
\end{equation*}
for some absolute constant $c > 0$ depending on $M$ only. One may take 
$c = c_M^{1/3} \min(c_0, c_6)/2$ with $c_M$ from \eqref{Absch}.

This example is related to a general framework of symmetric functions introduced 
by G\"otze, Naumov and Ulyanov \cite{G-N-U}. Indeed, we may consider
sequences of real functions $h_n(\theta_, \ldots, \theta_n)$,
defined on $\mathbb{R}^n$ such that
\begin{align}
\label{symmCond}
\begin{split}
 & h_{n+1} (\theta_1, \ldots, \theta_j, 0, \theta_{j+1}, \ldots, \theta_n) 
 = h_n (\theta_1, \ldots, \theta_j, \theta_{j+1}, \ldots, \theta_n);\\
 & \frac{\partial}{\partial \theta_j} 
   h_n (\theta_1, \ldots, \theta_j, \ldots, \theta_n) 
	 \Big|_{\theta_j = 0} = 0\enskip \ \ \forall j = 1, \ldots, n;\\
 & h_n(\theta_{\pi(1)}, \ldots, \theta_{\pi(n)}) 
 = h_n(\theta_1, \ldots, \theta_n)\enskip \forall \pi \in S_n,
\end{split}
\end{align}
where $S_n$ denotes the symmetric group. This model may be regarded as 
a general scheme which holds true in a lot of situations where asymptotic 
expansions are considered. As shown in \cite{G-N-U}, Conditions \eqref{symmCond} 
ensure the existence of a ``limit'' function 
$h_\infty(\sum_i \theta_i^2, \lambda_1, \ldots,\lambda_s)$, $s \in \mathbb{N}_0$, 
together with ``Edgeworth-type'' asymptotic expansions. These expansions are 
given in terms of polynomials of $Q_d(\theta)$, $d \ge 3$, where 
$Q_d(\theta) = Q_{d,a^*}$ for $a^* = (1, \ldots, 1)$, and coefficients given 
by the derivatives of the limit function $h_\infty$ at 
$\lambda_1 = \ldots = \lambda_s = 0$. In particular, if we assume 
$\theta \in S^{n-1}$, applying Proposition \ref{elementarePolynome} yields 
higher order concentration results for $(h_n)_n$.

In a certain sense, this represents a higher order extension of the second 
order results by Klartag \cite{K} for distribution functions of spherical 
weighted sums for log-concave measures. See also the results by Klartag 
and Sodin \cite{K-S} for sums of independent randoms variables.

\end{document}